\documentclass[11pt,leqno]{article}

\usepackage[active]{srcltx}

\usepackage{amsfonts,latexsym,amsmath,amssymb,amsthm}
\usepackage{fullpage}
\newtheorem{theorem}{Theorem}[section]
\newtheorem{lemma}[theorem]{Lemma}

\newtheorem{proposition}[theorem]{Proposition}
\newtheorem{corollary}[theorem]{Corollary}
\newtheorem{remark}[theorem]{Remark}
\newtheorem{example}[theorem]{Example}
\theoremstyle{definition}

\theoremstyle{remark}

\newtheorem*{note*}{Note}

\numberwithin{equation}{section}

\makeatletter
\newcommand{\rank}{\mathop{\operator@font rank}}
\newcommand{\conv}{\mathop{\operator@font conv}}
\newcommand{\vol}{\mathop{\operator@font vol}}
\newcommand{\onetagright}{\tagsleft@false}
\makeatother

\newcommand{\ls}{\leqslant}
\newcommand{\gr}{\geqslant}
\renewcommand{\epsilon}{\varepsilon}

\usepackage{color}

\newcommand{\set}[1]{\left\lbrace#1\right\rbrace}

\usepackage{hyperref}

\begin{document}
\small

\title{\bf Threshold for the expected measure of random polytopes}

\medskip

\author{Silouanos Brazitikos, Apostolos Giannopoulos and Minas Pafis}

\date{}

\maketitle

\begin{abstract}Let $\mu$ be a log-concave probability measure on ${\mathbb R}^n$ and for any $N>n$
consider the random polytope $K_N={\rm conv}\{X_1,\ldots ,X_N\}$, where $X_1,X_2,\ldots $ are independent random points in ${\mathbb R}^n$ distributed according to $\mu $. We study the question if there exists a threshold for the expected measure of $K_N$.
Our approach is based on the Cramer transform $\Lambda_{\mu}^{\ast }$ of $\mu $. We examine the existence of moments of all orders
for $\Lambda_{\mu}^{\ast }$ and establish, under some conditions, a sharp threshold for the expectation ${\mathbb E}_{\mu^N}[\mu (K_N)]$
of the measure of $K_N$: it is close to $0$ if $\ln N\ll {\mathbb E}_{\mu }(\Lambda_{\mu}^{\ast })$ and close to $1$ if $\ln N\gg {\mathbb E}_{\mu }(\Lambda_{\mu}^{\ast })$. The main condition is that the parameter $\beta(\mu)={\rm Var}_{\mu }(\Lambda_{\mu}^{\ast })/({\mathbb E}_{\mu }(\Lambda_{\mu }^{\ast }))^2$ should be small.
\end{abstract}

\section{Introduction}\label{section-1}

We study the question how to obtain a threshold for the expected measure of a random polytope
defined as the convex hull of independent random points with a log-concave distribution. The general
formulation of the problem is the following. Given a log-concave probability measure $\mu $ on ${\mathbb R}^n$,
let $X_1,X_2,\ldots $ be independent random points in ${\mathbb R}^n$ distributed according to $\mu $ and for any $N>n$ define the random polytope
$$K_N={\rm conv}\{X_1,\ldots ,X_N\}.$$
Then, consider the expectation ${\mathbb E}_{\mu^N}[\mu (K_N)]$ of the measure of $K_N$, where $\mu^N=\mu\otimes\cdots\otimes\mu $
($N$ times). This is an affinely invariant quantity, so we may assume that $\mu$ is centered, i.e. the barycenter of $\mu $ is at the origin.

Given $\delta\in (0,1)$ we say that $\mu $ satisfies a ``$\delta $-upper threshold" with constant $\varrho_1$ if
\begin{equation}\label{eq:kappa-1}\sup\{{\mathbb E}_{\mu^N}[\mu (K_N)]:N\ls \exp (\varrho_1n)\}\ls \delta \end{equation}
and that $\mu $ satisfies a ``$\delta $-lower threshold" with constant $\varrho_2$ if
\begin{equation}\label{eq:kappa-2}\inf\{{\mathbb E}_{\mu^N}[\mu (K_N)]:N\gr \exp (\varrho_2n)\}\gr 1-\delta .\end{equation}
Then, we define $\varrho_1(\mu ,\delta)=\sup\{\varrho_1:\eqref{eq:kappa-1}\; \textrm{holds true}\}$ and
$\varrho_2(\mu ,\delta )=\inf\{\varrho_2:\eqref{eq:kappa-2}\; \textrm{holds true}\}$. Our main goal is to obtain upper bounds
for the difference
$$\varrho (\mu ,\delta ):=\varrho_2(\mu ,\delta )-\varrho_1(\mu ,\delta )$$
for any fixed $\delta\in \left(0,\tfrac{1}{2}\right)$.

One may also consider a sequence $\{\mu_n\}_{n=1}^{\infty }$ of log-concave probability measures $\mu_n$ on ${\mathbb R}^n$.
Then, we say that $\{\mu_n\}_{n=1}^{\infty }$ exhibits a ``sharp threshold" if there exists a sequence $\{\delta_n\}_{n=1}^{\infty }$
of positive reals such that $\delta_n\to 0$ and $\varrho (\mu_n,\delta_n)\to 0$ as $n\to\infty $. This terminology may be used
to describe a variety of results that have been obtained for specific sequences of measures (in most cases, product measures or rotationally invariant measures). In Section~\ref{section-2} we provide a non-exhaustive list of contributions to this topic; starting with the classical
work \cite{DFM} of Dyer, F\"{u}redi and McDiarmid, which concerns the uniform measure on the cube, most of them establish a sharp threshold.

Our aim is to describe a general approach to the problem, working with an arbitrary log-concave probability measure $\mu $ on ${\mathbb R}^n$.
Our approach is based on the Cramer transform of $\mu $. Recall that the logarithmic Laplace transform of $\mu $ is defined by
\begin{equation*}\Lambda_{\mu }(\xi )=\ln\Big(\int_{{\mathbb R}^n}e^{\langle\xi
,z\rangle }d\mu (z)\Big),\qquad \xi\in {\mathbb R}^n\end{equation*}
and the Cramer transform of $\mu $ is the Legendre transform of $\Lambda_{\mu }$, defined by
\begin{equation*}\Lambda_{\mu }^{\ast }(x)= \sup_{\xi\in {\mathbb R}^n} \left\{ \langle x, \xi\rangle -\Lambda_{\mu }(\xi )\right\},
\qquad x\in {\mathbb R}^n.\end{equation*}
For every $t>0$ we set
\begin{equation*}B_t(\mu):=\{x\in{\mathbb R}^n:\Lambda^{\ast}_{\mu}(x)\ls t\}\end{equation*}
and for any $x\in {\mathbb R}^n$ we denote by ${\cal H}(x)$ the set
of all half-spaces $H$ of ${\mathbb R}^n$ containing $x$. Then we consider the function $\varphi_{\mu }$,
called Tukey's half-space depth, defined by
$$\varphi_{\mu }(x)=\inf\{\mu (H):H\in {\cal H}(x)\}.$$
We refer the reader to the survey article of Nagy, Sch\"{u}tt and Werner
\cite{Nagy-Schutt-Werner-2019} for an extensive and comprehensive survey on Tukey's half-space depth, with an emphasis on its connections
with convex geometry, and many references. From the definition of $\Lambda_{\mu}^{\ast }$ one can easily check that for every $x\in {\mathbb R}^n$ we have
$\varphi_{\mu } (x)\ls\exp (-\Lambda_{\mu }^{\ast }(x))$ (see Lemma~\ref{lem:upper-bt} in Section~\ref{section-3} below). In particular,
for any $t>0$ and for all $x\notin B_t(\mu )$ we have that $\varphi_{\mu }(x)\ls\exp (-t)$.
A main idea, which appears in all the previous works on this topic, is to show that $\varphi_{\mu }$ is almost constant
on the boundary $\partial (B_t(\mu))$ of $B_t(\mu)$. Our first main result shows that this is true, in general, if $\mu =\mu_K$
is the uniform measure on a centered convex body of volume $1$ in ${\mathbb R}^n$.

\begin{theorem}\label{th:intro-1}Let $K$ be a centered convex body of volume $1$ in ${\mathbb R}^n$.
Then, for every $t>0$ we have that
$$\inf\{\varphi_{\mu_K}(x):x\in B_t(\mu_K)\}\gr \frac{1}{10}\exp (-t-2\sqrt{n}).$$
This implies that
$$\omega_{\mu_K} (x)-5\sqrt{n}\ls \Lambda^{\ast }(x)\ls \omega_{\mu_K}(x)$$
for every $x\in {\mathbb R}^n$, where $\omega_{\mu_K}(x)=\ln\left(\frac{1}{\varphi_{\mu_K}(x)}\right)$.
\end{theorem}

Theorem~\ref{th:intro-1} may be viewed as a version of Cram\'{e}r's theorem (see \cite{Dembo-Zeitouni-book}) for random vectors uniformly distributed in convex bodies. We present the proof in Section~\ref{section-3}. It exploits techniques from the theory of large deviations and a theorem of Nguyen \cite{Nguyen-2014} (proved independently by Wang \cite{Wang-2014}; see also \cite{Fradelizi-Madiman-Wang-2016}) which is exactly the ingredient that forces us to consider only uniform measures on convex bodies. It seems harder to prove, if true, an analogous estimate for any centered log-concave probability measure $\mu $ on ${\mathbb R}^n$; this
is a basic question that our work leaves open.

The second step in our approach is to consider, for any centered log-concave probability measure $\mu $ on ${\mathbb R}^n$, the parameter
\begin{equation}\label{eq:intro-beta}\beta(\mu)=\frac{{\rm Var}_{\mu }(\Lambda_{\mu}^{\ast })}{({\mathbb E}_{\mu }(\Lambda_{\mu }^{\ast }))^2}\end{equation}
provided that $$\|\Lambda_{\mu }^{\ast }\|_{L^2(\mu )}=\big({\mathbb E}_{\mu }\big((\Lambda_{\mu }^{\ast })^2\big)\big )^{1/2}<\infty .$$
Roughly speaking, the plan is the following: provided that $\varphi_{\mu }$ is ``almost constant" on $\partial (B_t(\mu))$ for all $t>0$ and that
$\beta (\mu )=o_n(1)$, one can establish a ``sharp threshold" for the expected measure of $K_N$ with
$$\varrho_2\approx \varrho_1\approx \|\Lambda_{\mu }^{\ast }\|_{L^1(\mu )}={\mathbb E}_{\mu }(\Lambda_{\mu}^{\ast }).$$
We make these ideas more precise in Section~\ref{section-5} where we also illustrate them with a number of examples.
Note that it is not clear in advance that $\Lambda_{\mu}^{\ast }$ has bounded second or higher order moments, which is necessary so
that $\beta(\mu)$ would be well-defined. We study this question in
Section~\ref{section-4} where we obtain an affirmative answer in the case of the uniform measure on a convex body. In fact we cover
the more general case of $\kappa $-concave probability measures, $\kappa\in (0,1/n]$, which are supported on a centered convex body.

\begin{theorem}\label{th:intro-2}Let $K$ be a centered convex body of volume $1$ in ${\mathbb R}^n$. Let $\kappa\in (0,1/n]$ and let
$\mu$ be a centered $\kappa$-concave probability measure with ${\rm supp}(\mu)=K$. Then,
$$\int_{{\mathbb R}^n}e^{\frac{\kappa \Lambda_{\mu}^{\ast }(x)}{2}}d\mu(x)<\infty .$$
In particular, for all $p\gr 1$ we have that ${\mathbb E}_{\mu}\big((\Lambda_{\mu}^{\ast }(x))^p\big)<\infty $.
\end{theorem}

The method of proof of Theorem~\ref{th:intro-2} gives in fact reasonable upper bounds for $\|\Lambda_{\mu}^{\ast }\|_{L^p(\mu)}$.
In particular, if we assume that $\mu=\mu_K$ is the uniform measure on a centered convex body then we obtain a sharp two sided estimate
for the most interesting case where $p=1$ or $2$.

\begin{theorem}\label{th:intro-3}Let $K$ be a centered convex body of volume $1$ in ${\mathbb R}^n$, $n\gr 2$. Then,
$$c_1n/L_{\mu_K}^2\ls \|\Lambda_{\mu_K}^{\ast }\|_{L^1(\mu_K)}\ls \|\Lambda_{\mu_K}^{\ast }\|_{L^2(\mu_K)}\ls c_2n\ln n,$$
where $L_{\mu_K}$ is the isotropic constant of the uniform measure $\mu_K$ on $K$ and $c_1,c_2>0$ are absolute constants.
\end{theorem}

The left-hand side inequality of Theorem~\ref{th:intro-3} follows easily from one of the main results in \cite{BGP}. Both
the lower and the upper bound are of optimal order with respect to the dimension. This can be seen e.g. from the example of the
uniform measure on the cube or the Euclidean ball (see Section~\ref{section-5}), respectively.

Besides Theorem~\ref{th:intro-2}, we show in Section~\ref{section-4} that $\Lambda_{\mu }^{\ast }$ has finite moments of all orders
in the following cases:
\begin{enumerate}
\item[(i)] If $\mu $ is a centered probability measure on ${\mathbb R}$ which is absolutely continuous with respect to Lebesgue measure
or a product of such measures.
\item[(ii)] If $\mu $ is a centered log-concave probability measure on ${\mathbb R}^n$ and
there exists a function $g:[1,\infty )\to [1,\infty )$ with $\lim_{t\to\infty }g(t)/\ln (t+1)=+\infty $
such that $Z_t^+(\mu )\supseteq g(t)Z_2^+(\mu )$ for all $t\gr 2$, where $\{Z_t^+(\mu)\}_{t\gr 1}$ is the family
of one-sided $L_t$-centroid bodies of $\mu$.
\end{enumerate}
Again, it seems harder to prove, if true, an analogous result for any centered log-concave probability measure $\mu $ on ${\mathbb R}^n$; this
is a second basic question that our work leaves open.

In Section~\ref{section-5} we describe the approach to the main problem  and show how one can use the previous results to obtain bounds
for $\varrho(\mu,\delta)$. We also clarify the role of the parameter $\beta(\mu)$. One would hope that $\beta(\mu)$
is small as the dimension increases, e.g. $\beta(\mu)\ls c/\sqrt{n}$. If so, then the next general result provides
satisfactory lower bounds for $\varrho_1(\mu,\delta)$.

\begin{theorem}\label{th:intro-4}Let $\mu$ be a centered log-concave probability measure on ${\mathbb R}^n$.
Assume that $\beta (\mu )<1/8$ and $8\beta(\mu )<\delta <1$. If $n/L_{\mu}^2\gr c_2\ln (2/\delta )\sqrt{\delta /\beta(\mu)}$
where $L_{\mu}$ is the isotropic constant of $\mu $, then
$$\varrho_1(\mu ,\delta )\gr \left(1-\sqrt{8\beta(\mu)/\delta }\right)\frac{\mathbb{E}_{\mu}(\Lambda_{\mu}^{*})}{n}.$$
\end{theorem}

We are able to give satisfactory upper bounds for $\varrho_2(\mu,\delta)$ in the case where $\mu=\mu_K$ is the
uniform measure on a centered convex body $K$ of volume $1$ in ${\mathbb R}^n$.

\begin{theorem}\label{th:intro-5}Let $K$ be a centered convex body of volume $1$ in ${\mathbb R}^n$.
Let $\beta (\mu_K )<1/2$ and $2\beta(\mu_K )<\delta <1$. If $n/L_{\mu_K}^2\gr c_2\ln (2/\delta )\sqrt{\delta /\beta(\mu_K)}$ then
$$\varrho_2(\mu_K ,\delta )\ls \left(1+\sqrt{8\beta(\mu_K)/\delta }\right)\frac{\mathbb{E}_{\mu_K}(\Lambda_{\mu_K}^{*})}{n}.$$
\end{theorem}

Combining these two results we see that, provided that $\beta(\mu_K)$ is small compared to a fixed $\delta\in (0,1)$,
we have a threshold of the order
$$\varrho(\mu_K,\delta)\ls \frac{c}{n}\sqrt{\frac{{\rm Var}_{\mu_K }(\Lambda_{\mu_K}^{\ast })}{\delta }}.$$
The above discussion leaves open a third basic question: to estimate
$$\beta_n:=\sup\{\beta(\mu):\mu\;\hbox{is a centered log-concave probability measure on}\;{\mathbb R}^n\}.$$
We illustrate the method that we develop in this work with a number of examples. We consider first
the standard examples of the uniform measure on the unit cube and the Gaussian measure.
As a direct consequence of our results, in both cases we obtain a bound $\varrho (\mu,\delta)\ls c(\delta)/\sqrt{n}$ for the threshold,
where $c(\delta)>0$ is a constant depending on $\delta $. Finally, we examine the case of the uniform measure on the
Euclidean ball $D_n$ of volume $1$ in ${\mathbb R}^n$ and obtain the following sharp threshold.

\begin{theorem}\label{th:intro-6}Let $D_n$ be the centered Euclidean ball of volume $1$ in ${\mathbb R}^n$. Then, the sequence
$\mu_n:=\mu_{D_n}$ exhibits a sharp threshold with $\varrho (\mu_n,\delta )\ls \frac{c}{\sqrt{\delta n}}$ and e.g. in the case
where $n$ is even we have that
$${\mathbb E}_{\mu_n}(\Lambda_{\mu_n}^{\ast })=\frac{(n+1)}{2}H_{\frac{n}{2}}+O(\sqrt{n})$$
as $n\to\infty $, where $H_m=\sum_{k=1}^m\frac{1}{k}$.
\end{theorem}

\section{Notation, background information and related literature}\label{section-2}

In this section we introduce notation and terminology that we use throughout this work, and provide background
information on convex bodies and log-concave probability measures. We write $\langle\cdot ,\cdot\rangle $
for the standard inner product in ${\mathbb R}^n$ and denote the Euclidean norm by $|\cdot |$. In what follows, $B_2^n$ is the Euclidean unit ball, $S^{n-1}$ is the unit sphere, and $\sigma $ is the rotationally invariant probability measure on $S^{n-1}$. Lebesgue measure in ${\mathbb R}^n$ is
also denoted by $|\cdot |$. The letters $c, c^{\prime },c_j,c_j^{\prime }$ etc. denote absolute positive constants whose value may change from line to line.
Whenever we write $a\approx b$, we mean that there exist absolute constants $c_1,c_2>0$ such that $c_1a\ls b\ls c_2a$.

We refer to Schneider's book \cite{Schneider-book} for basic facts from the Brunn-Minkowski theory and to the book
\cite{AGA-book} for basic facts from asymptotic convex geometry. We also refer to \cite{BGVV-book} for more information on isotropic
convex bodies and log-concave probability measures.

\subsection{Log-concave probability measures}

A convex body in ${\mathbb R}^n$ is a compact convex set $K\subset {\mathbb R}^n$ with non-empty interior.
We say that $K$ is centrally symmetric if $-K=K$ and that $K$ is centered if the barycenter ${\rm bar}(K)=\frac{1}{|K|}\int_Kx\,dx$ of $K$ is at the origin.
The Minkowski functional $\|\cdot\|_K$ of a convex body $K$ in ${\mathbb R}^n$ with $0\in {\rm int}(K)$
is defined for all $x\in {\mathbb R}^n$ by $\|x\|_K=\inf \{s\gr 0:x\in sK\}$ and
the support function of $K$ is the function
$$h_K(x) = \sup\{\langle x,y\rangle :y\in K\},\qquad x\in {\mathbb R}^n.$$
A Borel measure $\mu$ on $\mathbb R^n$ is called log-concave if $\mu(\lambda
A+(1-\lambda)B) \gr \mu(A)^{\lambda}\mu(B)^{1-\lambda}$ for any compact subsets $A$
and $B$ of ${\mathbb R}^n$ and any $\lambda \in (0,1)$. A function
$f:\mathbb R^n \rightarrow [0,\infty)$ is called log-concave if
its support $\{f>0\}$ is a convex set in ${\mathbb R}^n$ and the restriction of $\ln{f}$ to it is concave.
If $f$ has finite positive integral then there exist constants $A,B>0$ such that $f(x)\ls Ae^{-B|x|}$
for all $x\in {\mathbb R}^n$ (see \cite[Lemma~2.2.1]{BGVV-book}). In particular, $f$ has finite moments
of all orders. It is known (see \cite{Borell-1974}) that if a probability measure $\mu $ is log-concave and $\mu (H)<1$ for every
hyperplane $H$ in ${\mathbb R}^n$, then $\mu $ has a log-concave density $f_{{\mu }}$.
We say that $\mu $ is even if $\mu (-B)=\mu (B)$ for every Borel subset $B$ of ${\mathbb R}^n$ and that $\mu $ is centered if
\begin{equation*}
{\rm bar}(\mu):=\int_{\mathbb R^n} \langle x, \xi \rangle d\mu(x) = \int_{\mathbb R^n} \langle x, \xi \rangle f_{\mu}(x) dx = 0
\end{equation*} for all $\xi\in S^{n-1}$. We shall use the fact that if $\mu $ is a centered log-concave
probability measure on ${\mathbb R}^k$ then
\begin{equation}\label{eq:frad-2}\|f_{\mu }\|_{\infty }\ls e^kf_{\mu }(0).\end{equation}
This is a result of Fradelizi from \cite{Fradelizi-1997}. Note that if $K$ is a convex body in
$\mathbb R^n$ then the Brunn-Minkowski inequality implies that the indicator function
$\mathbf{1}_{K} $ of $K$ is the density of a log-concave measure, the Lebesgue measure on $K$.

Given $\kappa\in [-\infty ,1/n]$ we say that a measure $\mu $ on
${\mathbb R}^n$ is $\kappa $-concave if
\begin{equation}\label{eq:s-concave-measure-1}\mu ((1-\lambda )A+\lambda B)\gr
((1-\lambda )\mu^{\kappa }(A)+\lambda \mu^{\kappa }(B))^{1/\kappa }\end{equation}
for all compact subsets $A,B$ of ${\mathbb R}^n$ with $\mu (A)\mu (B)>0$ and all $\lambda\in (0,1)$.
The limiting cases are defined appropriately. For $\kappa =0$ the right hand side in \eqref{eq:s-concave-measure-1} becomes
$\mu(A)^{1-\lambda }\mu (B)^{\lambda }$ (therefore, $0$-concave measures are the log-concave measures).
In the case $\kappa =-\infty $ the right hand side in \eqref{eq:s-concave-measure-1} becomes $\min\{\mu (A),\mu (B)\}$.
Note that if $\mu $ is $\kappa $-concave and $\kappa_1\ls \kappa$ then
$\mu $ is $\kappa_1$-concave.

Next, let $\gamma \in [-\infty ,\infty ]$. A function $f:\mathbb R^n\to [0,\infty)$ is called $\gamma $-concave
if
\begin{equation*}f((1-\lambda )x+\lambda y)\gr
((1-\lambda )f^{\gamma }(x)+\lambda f^{\gamma }(y))^{1/\gamma }\end{equation*}
for all $x,y\in {\mathbb R}^n$ with $f(x)f(y)>0$ and all $\lambda\in (0,1)$. Again, we define the cases $\gamma =0,+\infty $
appropriately. Borell \cite{Borell-1975} studied the relation between $\kappa $-concave probability
measures and $\gamma $-concave functions and showed that if $\mu $ is a measure on ${\mathbb R}^n$ and the affine subspace $F$ spanned by the support
${\rm supp}(\mu )$ of $\mu $ has dimension ${\rm dim}(F)=n$ then for every $-\infty \ls \kappa <1/n$ we have that
$\mu $ is $\kappa $-concave if and only if it has a non-negative density $\psi\in L_{{\rm loc}}^1({\mathbb R}^n,dx)$
and $\psi $ is $\gamma $-concave, where $\gamma =\frac{\kappa}{1-\kappa n}\in [-1/n,+\infty )$.

Let $\mu $ and $\nu$ be two log-concave probability measures on ${\mathbb R}^n$. Let $T:{\mathbb
R}^n\to {\mathbb R}^n$ be a measurable function which is defined
$\nu $-almost everywhere and satisfies
\begin{equation*}\mu (B)=\nu (T^{-1}(B))\end{equation*}
for every Borel subset $B$ of ${\mathbb R}^n$. We then say that $T$
pushes forward $\nu $ to $\mu $ and write $T_*\nu=\mu$. It is
easy to see that $T_*\nu =\mu $ if and only if for every bounded Borel
measurable function $g:{\mathbb R}^n\to {\mathbb R}$ we have
\begin{equation*}\int_{{\mathbb R}^n}g(x)d\mu (x)=\int_{{\mathbb R}^n}g(T(y))d\nu (y).\end{equation*}
If $\mu $ is a log-concave measure on ${\mathbb R}^n$ with density $f_{\mu}$, we define the isotropic constant of $\mu $ by
\begin{equation*}
L_{\mu }:=\left (\frac{\sup_{x\in {\mathbb R}^n} f_{\mu} (x)}{\int_{{\mathbb
R}^n}f_{\mu}(x)dx}\right )^{\frac{1}{n}} [\det \textrm{Cov}(\mu)]^{\frac{1}{2n}},\end{equation*}
where $\textrm{Cov}(\mu)$ is the covariance matrix of $\mu$ with entries
\begin{equation*}\textrm{Cov}(\mu )_{ij}:=\frac{\int_{{\mathbb R}^n}x_ix_j f_{\mu}
(x)\,dx}{\int_{{\mathbb R}^n} f_{\mu} (x)\,dx}-\frac{\int_{{\mathbb
R}^n}x_i f_{\mu} (x)\,dx}{\int_{{\mathbb R}^n} f_{\mu}
(x)\,dx}\frac{\int_{{\mathbb R}^n}x_j f_{\mu}
(x)\,dx}{\int_{{\mathbb R}^n} f_{\mu} (x)\,dx}.\end{equation*} We say
that a log-concave probability measure $\mu $ on ${\mathbb R}^n$
is isotropic if it is centered and $\textrm{Cov}(\mu )=I_n$,
where $I_n$ is the identity $n\times n$ matrix. In this case, $L_{\mu }=\|f_{\mu }\|_{\infty }^{1/n}$.
For every $\mu $ there exists an affine transformation $T$
such that $T_{\ast }\mu $ is isotropic. The hyperplane conjecture asks if there exists an absolute constant $C>0$ such that
\begin{equation*}L_n:= \max\{ L_{\mu }:\mu\ \hbox{is an isotropic log-concave probability measure on}\ {\mathbb R}^n\}\ls C\end{equation*}
for all $n\gr 1$.  Bourgain \cite{Bourgain-1991}
established  the upper bound $L_n\ls c\sqrt[4]{n}\ln n$; later, Klartag, in \cite{Klartag-2006},
improved this estimate to $L_n\ls c\sqrt[4]{n}$. In a breakthrough work, Chen \cite{C} proved that for any $\epsilon >0$
there exists $n_0(\epsilon )\in {\mathbb N}$ such that $L_n\ls n^{\epsilon}$ for every $n\gr n_0(\epsilon )$. Very recently,
Klartag and Lehec \cite{Klartag-Lehec-2022} showed that the hyperplane conjecture and the stronger Kannan-Lov\'{a}sz-Simonovits
isoperimetric conjecture hold true up to a factor that is polylogarithmic in the dimension; more precisely, they achieved
the bound $L_n\ls c(\ln n)^4$, where $c>0$ is an absolute constant.

\subsection{Known results}

Several variants of the threshold problem have been studied, starting with the work of Dyer, F\"{u}redi and McDiarmid
who established in \cite{DFM} a sharp threshold
for the expected volume of random polytopes with vertices uniformly distributed in the discrete cube $E_2^n=\{-1,1\}^n$
or in the solid cube $B_{\infty }^n=[-1,1]^n$. For example, in the first case, if $\kappa =2/\sqrt{e}$ then for every
$\epsilon \in (0,1)$ one has
$$\lim_{n\rightarrow\infty }\sup\left\{2^{-n} {\mathbb E}|K_N|\colon N\ls (\kappa -\epsilon
)^n\right\}=0$$ and $$\lim_{n\rightarrow\infty
}\inf\left \{ 2^{-n} {\mathbb E}|K_N|\colon N\gr (\kappa
+\epsilon )^n\right\}=1.$$
A similar result holds true for the expected volume of random polytopes with vertices uniformly distributed in
the cube $B_{\infty }^n$; the corresponding value of the
constant $\kappa $ is $\kappa =2\pi /e^{\gamma +1/2}$, where $\gamma $ is Euler's constant. In the terminology
of the introduction, this last result is equivalent to $\varrho (\mu_{B_{\infty }^n},\delta)=O_{\delta }(1)$
for any fixed value of $\delta\in \left(0,\tfrac{1}{2}\right)$.

Further sharp thresholds for the volume of various classes of random polytopes have been given. In \cite{Gatzouras-Giannopoulos-2009}
a threshold for ${\mathbb E}_{\mu^N}|K_N|/(2\alpha)^n$ was established for the case where $X_i$ have independent identically
distributed coordinates supported on a bounded interval $[-\alpha ,\alpha ]$ under some mild additional assumptions.
The articles \cite{Pivovarov-2007} and \cite{Bonnet-Chasapis-Grote-Temesvari-Turchi-2019}, \cite{Bonnet-Kabluchko-Turchi-2021}
address the same question for a number of cases where $X_i$ have rotationally invariant densities. Exponential
in the dimension upper and lower thresholds are obtained in \cite{Frieze-Pegden-Tkocz-2020} for the case where $X_i$
are uniformly distributed in a simplex.

Upper and lower thresholds were obtained recently by Chakraborti, Tkocz and Vritsiou in \cite{Chakraborti-Tkocz-Vritsiou-2021} for some
general families of distributions. If $\mu $ is an even log-concave probability measure supported on a convex body $K$ in ${\mathbb R}^n$
and if $X_1,X_2,\ldots $ are independent random points distributed according to $\mu $, then for any $n<N\ls \exp (c_1n/L_{\mu }^2)$ we have that
\begin{equation}\label{eq:tk-1}\frac{{\mathbb E}_{\mu^N}(|K_N|)}{|K|} \ls \exp\left(-c_2n/L_{\mu }^2\right),\end{equation}
where $c_1,c_2>0$ are absolute constants. A lower threshold is also established in \cite{Chakraborti-Tkocz-Vritsiou-2021} for the case
where $\mu $ is an even $\kappa $-concave
measure on ${\mathbb R}^n$ with $0<\kappa <1/n$, supported on a convex body $K$ in ${\mathbb R}^n$. If $X_1,X_2,\ldots $ are independent
random points in ${\mathbb R}^n$ distributed according to $\mu $ and $K_N={\rm conv}\{X_1,\ldots ,X_N\}$ as before, then for any $M\gr C$ and any
$N\gr \exp\left(\frac{1}{\kappa }(\log n+2\log M)\right)$ we have that
\begin{equation}\label{eq:tk-2}\frac{{\mathbb E}_{\mu^N}(|K_N|)}{|K|}\gr 1-\frac{1}{M},\end{equation}
where $C>0$ is an absolute constant.

Analogues of these results in the setting of the present work were obtained in \cite{BGP} for $0$-concave, i.e. log-concave, probability measures.
There exists an absolute constant $c>0$ such that if $N_1(n)=\exp (cn/L_n^2)$ then
$$\sup_{\mu }\Big(\sup\Big\{{\mathbb E}_{\mu^N}[\mu (K_N)]:N\ls N_1(n)\Big\}\Big)\longrightarrow 0$$
as $n\to\infty $, where the first supremum is over all log-concave probability measures $\mu $ on ${\mathbb R}^n$.
On the other hand, if $\delta\in (0,1)$ then,
$$\inf_{\mu }\Big(\inf\Big\{ {\mathbb E}_{\mu^N}\big[\mu ((1+\delta )K_N)\big]: N\gr \exp \big (C\delta^{-1}\ln \left(2/\delta \right)n\ln n\big )\Big\}\Big)\longrightarrow 1$$ as $n\to\infty $, where the first infimum is over all log-concave probability measures $\mu $ on ${\mathbb R}^n$ and $C>0$
is an absolute constant.

It should be noted that an exponential in the dimension lower threshold is not possible in full generality. For example, in the case
where $X_i$ are uniformly distributed in the Euclidean ball one needs $N\gr\exp (cn\ln n)$ points so that the volume of a random
$K_N$ will be significantly large. Thus, apart from the constants depending on $\delta $, the lower threshold above is sharp.
However, it provides a weak threshold in the sense that we estimate the
expectation ${\mathbb E}_{\mu^N}\big(\mu (1+\delta )K_N)$ (for an arbitrarily small but positive value of $\delta )$ while
we would like to have a similar result for ${\mathbb E}_{\mu^N}\big[\mu (K_N)]$. It is shown in \cite{BGP}
that there exists an absolute constant $C>0$ such that
$$\inf_{\mu }\Big(\inf\Big\{ {\mathbb E}\,\big[\mu (K_N)\big]: N\gr \exp (C(n\ln n)^2u(n))\Big\}\Big)\longrightarrow 1$$
as $n\to\infty $, where the first infimum is over all log-concave probability measures $\mu $ on ${\mathbb R}^n$
and $u(n)$ is any function with $u(n)\to\infty $ as $n\to\infty $.

\section{Estimates for the half-space depth}\label{section-3}

Let $\mu $ be a centered log-concave probability measure on ${\mathbb R}^n$ with density $f:=f_{\mu }$.
Recall that the logarithmic Laplace transform of $\mu $ on ${\mathbb R}^n$ is defined by
\begin{equation*}\Lambda_{\mu }(\xi )=\log\Big(\int_{{\mathbb R}^n}e^{\langle\xi
,z\rangle }f(z)dz\Big).\end{equation*} It is easily checked by means of H\"{o}lder's inequality that
$\Lambda_{\mu }$ is convex and $\Lambda_{\mu }(0)=0$. Since ${\rm
bar}(\mu )=0$, Jensen's inequality shows that
\begin{equation*}\Lambda_{\mu }(\xi )=\log\Big(\int_{{\mathbb R}^n}e^{\langle\xi
,z\rangle }f(z)dz\Big)\gr \int_{{\mathbb R}^n}\langle \xi,z\rangle f(z)dz=0\end{equation*} for all $\xi $. Therefore, $\Lambda_{\mu }$ is
a non-negative function. One can check that the set $A(\mu )=\{\Lambda_{\mu }<\infty\}$ is open and $\Lambda_{\mu }$
is $C^{\infty }$ and strictly convex on $A(\mu )$.

We also define
\begin{equation*}\Lambda_{\mu }^{\ast }(x)= \sup_{\xi\in {\mathbb R}^n} \left\{ \langle x, \xi\rangle - \Lambda_{\mu }(\xi )\right\}.\end{equation*}
In other words, $\Lambda_{\mu}^{\ast}$ is the Legendre transform of $\Lambda_{\mu }$: recall that
given a convex function $g:{\mathbb R}^n\to (-\infty ,\infty ]$, the Legendre transform ${\cal L}(g)$
of $g$ is defined by
\begin{equation*}{\cal L}(g)(x):=\sup_{\xi\in {\mathbb R}^n}\{ \langle x,\xi \rangle -
g(\xi )\}.\end{equation*}The function $\Lambda^{\ast }_{\mu
}$ is called the Cramer transform of $\mu $. For every $t>0$ we define the convex set
\begin{equation*}B_t(\mu):=\{x\in{\mathbb R}^n:\Lambda^{\ast}_{\mu}(x)\ls t\}.\end{equation*}
For any $x\in {\mathbb R}^n$ we denote by ${\cal H}(x)$ the set
of all half-spaces $H$ of ${\mathbb R}^n$ containing $x$. Then we define
$$\varphi_{\mu }(x)=\inf\{\mu (H):H\in {\cal H}(x)\}.$$
The function $\varphi_{\mu }$ is called Tukey's half-space depth. Our aim is to give upper and lower bounds
for $\varphi_{\mu }(x)$ when $x\in\partial (B_t(\mu ))$, $t>0$.

\begin{lemma}\label{lem:upper-bt}Let $\mu $ be a Borel probability measure on ${\mathbb R}^n$.
For every $x\in {\mathbb R}^n$ we have $\varphi_{\mu } (x)\ls\exp (-\Lambda_{\mu }^{\ast }(x))$. In particular,
for any $t>0$ and for all $x\notin B_t(\mu )$ we have that $\varphi_{\mu }(x)\ls\exp (-t)$.
\end{lemma}

\begin{proof}Let $x\in {\mathbb R}^n$. We start with the observation that for any $\xi \in {\mathbb R}^n$ the half-space
$\{z:\langle z-x,\xi  \rangle \gr 0\}$ is in ${\cal H}(x)$, therefore
\begin{equation*}
\varphi_{\mu } (x) \ls \mu (\{z:\langle z,\xi  \rangle \gr \langle x,\xi \rangle\} )
\ls e^{-\langle x,\xi \rangle }{\mathbb E}_{\mu }\big(e^{\langle z,\xi \rangle }\big)=\exp \big(-[\langle x,\xi \rangle -\Lambda_{\mu }(\xi )]\big),\end{equation*}
and taking the infimum over all $\xi \in {\mathbb R}^n$ we see that
$\varphi_{\mu } (x)\ls\exp (-\Lambda_{\mu }^{\ast }(x))$, as claimed.
\end{proof}

Next, we would like to obtain a lower bound for $\varphi_{\mu }(x)$ when $x\in B_t(\mu )$. In the case where $\mu=\mu_K$
is the uniform measure on a centered convex body $K$ of volume $1$ in ${\mathbb R}^n$, our estimate is the following.

\begin{theorem}\label{th:lower-bt-body}Let $K$ be a centered convex body of volume $1$ in ${\mathbb R}^n$.
Then, for every $t>0$ we have that
$$\inf\{\varphi_{\mu_K}(x):x\in B_t(\mu_K)\}\gr \frac{1}{10}\exp (-t-2\sqrt{n}).$$
\end{theorem}

The first part of the argument works for any centered log-concave probability measure $\mu$
with density $f$ on ${\mathbb R}^n$. For every $\xi\in {\mathbb R}^n$ we define the probability measure $\mu_{\xi }$ with density
$$f_{\xi }(z)=e^{-\Lambda_{\mu } (\xi )+\langle\xi, z\rangle}f(z).$$
In the next lemma (see \cite[Proposition~7.2.1]{BGVV-book}) we recall some basic facts for $\mu_{\xi }$.

\begin{lemma}The barycenter of $\mu_{\xi }$ is $x=\nabla\Lambda_{\mu } (\xi )$ and
${\rm Cov}(\mu_{\xi})={\rm Hess}\,(\Lambda_{\mu })(\xi)$.
\end{lemma}

Next, we set
$$\sigma_{\xi }^2=\int_{{\mathbb R}^n}\langle z-x,\xi\rangle^2d\mu_{\xi }(z).$$
Let $t>0$. Since $B_t(\mu)$ is convex, in order to give a lower bound for $\inf\{\varphi_{\mu }(x):x\in B_t(\mu)\}$
it suffices to give a lower bound for $\mu (H)$, where $H$ is any closed half-space whose bounding hyperplane
supports $B_t(\mu )$. In that case,
\begin{equation}\label{eq:lower-sigmaxi-1}\mu(H)=\mu(\{z:\langle z-x,\xi\rangle\gr 0\})\end{equation}
for some $x\in\partial(B_t(\mu ))$, with $\xi=\nabla\Lambda_{\mu}^{*}(x)$, or equivalently $x=\nabla\Lambda_{\mu}(\xi)$
(see e.g. Theorem~23.5 and Corollary~23.5.1 in \cite{Rockafellar-book}). Note that
\begin{align}\label{eq:lower-sigmaxi-2}
\mu(\{z:\langle z-x,\xi\rangle\gr 0\})&=\int_{\mathbb{R}^n}{\bf 1}_{[0,\infty)} (\langle z-x,\xi \rangle) f(z)\, dz\\
\nonumber &=e^{\Lambda_{\mu } (\xi )}\int_{\mathbb{R}^n}{\bf 1}_{[0,\infty)} (\langle z-x,\xi \rangle ) e^{-\langle z,\xi \rangle}\, d\mu_{\xi}(z)\\
\nonumber &=e^{\Lambda_{\mu } (\xi )}e^{-\langle x,\xi \rangle}\int_{\mathbb{R}^n}{\bf 1}_{[0,\infty)}(\langle z-x,\xi\rangle) e^{-\langle z-x,\xi \rangle}\, d\mu_{\xi }(z)\\
\nonumber &\gr e^{-\Lambda_{\mu }^{\ast }(x)}\int_{0}^{\infty} \sigma_{\xi }e^{-\sigma_{\xi }t}
\mu_{\xi}(\{z:0\ls\langle z-x,\xi \rangle\ls \sigma_{\xi }t\})\, dt.
\end{align}
From Markov's inequality we see that
$$\mu_{\xi}(\{z:\langle z-x,\xi \rangle\gr 2\sigma_{\xi }\})\ls \frac{1}{4}.$$ Moreover, since $x$ is the barycenter
of $\mu_{\xi }$, Gr\"{u}nbaum's lemma (see \cite[Lemma~2.2.6]{BGVV-book}) implies that
$$\mu_{\xi}(\{z:\langle z-x,\xi \rangle\gr 0\})\gr \frac{1}{e}.$$ Therefore,
\begin{equation}\label{eq:lower-sigmaxi-3}\int_{0}^{\infty} \sigma_{\xi }e^{-\sigma_{\xi }t}\mu_{\xi}(\{z:0\ls \langle z-x,\xi \rangle\ls \sigma_{\xi }t\})\, dt\gr\int_2^{\infty} \sigma_{\xi }e^{-\sigma_{\xi }t}\left(\frac{1}{e}-\frac{1}{4}\right)\, dt\gr\frac{4-e}{4e}e^{-2\sigma_{\xi }}.\end{equation}
We would like now an upper bound for $\sup_{\xi }\sigma_{\xi}$. We can have this when $\mu=\mu_K$ is the uniform measure on a
centered convex body $K$ of volume $1$ on ${\mathbb R}^n$, using a theorem of Nguyen \cite{Nguyen-2014} (proved
independently by Wang \cite{Wang-2014}; see also \cite{Fradelizi-Madiman-Wang-2016}).

\begin{theorem}\label{th:Nguyen}Let $\nu $ be a log-concave probability measure on ${\mathbb R}^n$ with density $g=\exp (-p)$,
where $p$ is a convex function. Then,
$${\rm Var}_{\nu }(p)\ls n.$$
\end{theorem}

\begin{proof}[Proof of Theorem~$\ref{th:lower-bt-body}$]Set $\mu :=\mu_K$. Since $f(z)={\mathbf 1}_K(z)$, the density $f_{\xi }$ of $\mu_{\xi }$ is proportional to $e^{\langle\xi, z\rangle}{\mathbf 1}_K(z)$. Using the fact that
$$\mathbb{E}_{\mu_{\xi}}({\langle\xi, z\rangle})=\langle \nabla\Lambda_{\mu } (\xi ),\xi\rangle =\langle x,\xi\rangle ,$$
from Theorem~\ref{th:Nguyen} we get that
$$\sigma_{\xi }^2=\mathbb{E}_{\mu_{\xi}}({\langle z-x,\xi\rangle})^2={\rm Var}_{\mu_{\xi }}(\langle\xi, z\rangle)\ls n.$$
Then, combining \eqref{eq:lower-sigmaxi-1}, \eqref{eq:lower-sigmaxi-2} and \eqref{eq:lower-sigmaxi-3},
for any bounding hyperplane $H$ of $B_t(\mu)$ we have
\begin{align*}\mu (H) &\gr  e^{-\Lambda_{\mu }^{\ast }(x)}\int_{0}^{\infty} \sigma_{\xi }e^{-\sigma_{\xi }t}
\mu_{\xi}(0\ls\langle z-x,\xi \rangle\ls \sigma_{\xi }t)\, dt\\
&\gr \frac{4-e}{4e}e^{-\Lambda_{\mu}^{\ast }(x)-2\sigma_{\xi }}\gr \frac{1}{10}\exp (-t-2\sqrt{n}),
\end{align*}
as claimed.
\end{proof}

Theorem~\ref{th:lower-bt-body} shows that if $K$ is a centered convex body of volume $1$ in ${\mathbb R}^n$ then
$$10\varphi_{\mu_K}(x)\gr\exp (-\Lambda_{\mu_K}^{\ast }(x)-2\sqrt{n})$$
for all $x\in {\mathbb R}^n$. Setting
\begin{equation}\label{eq:omega}\omega_{\mu_K} (x)=\ln\left(\frac{1}{\varphi_{\mu_K}(x)}\right)\end{equation}
and taking into account Lemma~\ref{lem:upper-bt} we have the next two-sided estimate.

\begin{corollary}\label{cor:L*-omega}Let $K$ be a centered convex body of volume $1$ in ${\mathbb R}^n$.
Then, for every $x\in {\rm int}(K)$ we have that
\begin{equation}\label{eq:L*-omega}\omega_{\mu_K} (x)-5\sqrt{n}\ls \Lambda_{\mu_K}^{\ast }(x)\ls \omega_{\mu_K}(x).\end{equation}
\end{corollary}

\begin{note*}\rm A basic question that arises from the results of this section is whether
an analogue of \eqref{eq:L*-omega} holds true for any centered log-concave probability measure $\mu $ on ${\mathbb R}^n$.
This would allow us to apply the next steps of the procedure that our approach suggests to all log-concave
probability measures.
\end{note*}

\section{Moments of the Cramer transform}\label{section-4}

As explained in the introduction, we would like to know for which centered log-concave probability measures $\mu $ on ${\mathbb R}^n$
we have that $\Lambda_{\mu}^{\ast }$ has finite moments of all orders.
Our first result provides an affirmative answer in the case where $\mu=\mu_K$ is the uniform measure
on a centered convex body $K$ of volume $1$ in ${\mathbb R}^n$. In fact, the next theorem covers a more general case.

\begin{theorem}\label{th:moments-kappa}Let $K$ be a centered convex body of volume $1$ in ${\mathbb R}^n$. Let $\kappa\in (0,1/n]$ and let
$\mu$ be a centered $\kappa$-concave probability measure with ${\rm supp}(\mu)=K$. Then,
$$\int_{{\mathbb R}^n}e^{\frac{\kappa \Lambda_{\mu}^{\ast }(x)}{2}}d\mu(x)<\infty .$$
In particular, for all $p\gr 1$ we have that ${\mathbb E}_{\mu}\big((\Lambda_{\mu}^{\ast }(x))^p\big)<\infty $.
\end{theorem}

The proof of Theorem~\ref{th:moments-kappa} is based on the next lemma, which is proved in \cite[Lemma~7]{Chakraborti-Tkocz-Vritsiou-2021}
in the symmetric case.

\begin{lemma}\label{lem:ctv}Let $K$ be a centered convex body of volume $1$ in ${\mathbb R}^n$. Let $\kappa\in (0,1/n]$ and let $\mu$ be a centered $\kappa$-concave probability measure with ${\rm supp}(\mu)=K$. Then,
\begin{equation}\label{eq:ctv}\varphi_{\mu }(x)\gr e^{-2}\kappa (1-\|x\|_K)^{1/\kappa }\end{equation}
for every $x\in K$, where $\|x\|_K$ is the Minkowski functional of $K$.
\end{lemma}

\begin{proof}[Sketch of the proof]We modify the argument from \cite[Lemma~7]{Chakraborti-Tkocz-Vritsiou-2021}
to cover the not necessarily symmetric case. First, consider the case $0<\kappa <1/n$. Let $X$ be a random vector distributed according to $\mu $. Given $\theta\in S^{n-1}$ let $b=h_K(\theta)$ and $a=h_K(-\theta )$. If $g_{\theta}$ is the density of $\langle X,\theta\rangle $ then $g_{\theta}^{\frac{\kappa }{1-\kappa }}$ is concave
on $[-a,b]$, therefore
$$g_{\theta}(t)\gr g_{\theta}(0)\left(1-\frac{t}{b}\right)^{\frac{1-\kappa}{\kappa}}$$
for all $t\in [0,b]$. It follows that, for every $0<s<b$,
$${\mathbb P}(\langle X,\theta\rangle \gr s)=\int_s^bg_{\theta}(t)\,dt\gr g_{\theta}(0)\int_s^b\left(1-\frac{t}{b}\right)^{\frac{1-\kappa}{\kappa}}\,dt
=\kappa g_{\theta}(0)b\left(1-\frac{s}{b}\right)^{\frac{1}{\kappa}}.$$
Note that $g_{\theta}$ is a centered log-concave density. Therefore, $g_{\theta}(0)\gr e^{-1}\|g_{\theta}\|_{\infty }$
by \eqref{eq:frad-2} and $\|g_{\theta}\|_{\infty }b\gr {\mathbb P}(\langle X,\theta\rangle \gr 0)\gr e^{-1}$
by Gr\"{u}nbaum's lemma \cite[Lemma~2.2.6]{BGVV-book}, which implies that $g_{\theta}(0)b\gr e^{-2}$. It follows that
$${\mathbb P}(\langle X,\theta\rangle \gr s)=\int_s^bg_{\theta}(t)\,dt\gr e^{-2}\kappa \left(1-\frac{s}{b}\right)^{\frac{1}{\kappa}}.$$
Now, let $x\in K$. Then $\langle x,\theta\rangle\ls \|x\|_Kh_K(\theta)=\|x\|_Kb$, therefore
$${\mathbb P}(\langle X,\theta\rangle \gr \langle x,\theta\rangle )\gr {\mathbb P}(\langle X,\theta\rangle \gr \|x\|_Kb)
\gr e^{-2}\kappa \left(1-\|x\|_K\right)^{\frac{1}{\kappa}}.$$
For the case $\kappa =1/n$ recall that a $1/n$-concave measure is $\kappa $-concave for every $\kappa\in (0,1/n)$.
This means that \eqref{eq:ctv} holds true for all $\kappa\in (0,1/n)$ and letting $\kappa\to 1/n$ we obtain the result.
\end{proof}

\smallskip

\begin{proof}[Proof of Theorem~$\ref{th:moments-kappa}$]From Lemma~\ref{lem:upper-bt} we know that $\varphi_{\mu } (x)\ls\exp (-\Lambda_{\mu }^{\ast }(x))$,
or equivalently,
$$e^{\frac{\kappa \Lambda_{\mu}^{\ast }(x)}{2}}\ls \frac{1}{\varphi_{\mu}(x)^{\kappa /2}}$$
for all $x\in K$. From Lemma~\ref{lem:ctv} we know that
$$\varphi_{\mu }(x)\gr e^{-2}\kappa (1-\|x\|_K)^{1/\kappa }$$
for every $x\in K$. It follows that
$$\int_Ke^{\frac{\kappa \Lambda_{\mu}^{\ast }(x)}{2}}d\mu(x)\ls (e^2/\kappa)^{\kappa/2}\int_K\frac{1}{(1-\|x\|_K)^{1/2}}d\mu(x).$$
Recall that the cone probability measure $\nu_K$ on the
boundary $\partial (K)$ of a convex body $K$ with $0\in {\rm int}(K)$ is defined by
$$\nu_K(B)=\frac{|\{rx:x\in B,0\ls r\ls 1\}|}{|K|}$$
for all Borel subsets $B$ of $\partial (K)$. We shall use the identity
\begin{equation*}\int_{{\mathbb R}^n}g(x)\,dx=n|K|\int_0^{\infty }r^{n-1}\int_{\partial (K)}g(rx)\,d\nu_K(x)\,dr\end{equation*}
which holds for every integrable function $g:{\mathbb R}^n\to {\mathbb R}$ (see \cite[Proposition~1]{Naor-Romik}).
Let $f$ denote the density of $\mu $ on $K$. We write
\begin{align*}
\int_K\frac{1}{(1-\|x\|_K)^{1/2}}d\mu(x) &=\int_{{\mathbb R}^n}\frac{f(x)}{(1-\|x\|_K)^{1/2}}{\mathbf 1}_K(x)\,dx\\
&= n|K|\int_0^{\infty }r^{n-1}\int_{\partial (K)}\frac{f(ry)}{(1-\|ry\|_K)^{1/2}}{\mathbf 1}_K(ry)\,d\nu_K(y)\,dr\\
&=n|K|\int_0^1\frac{r^{n-1}}{\sqrt{1-r}}\int_{\partial (K)}f(ry)\,d\nu_K(y)\,dr\\
&\ls n|K|\|f\|_{\infty }\int_0^1\frac{r^{n-1}}{\sqrt{1-r}}\,dr =n|K|B(n,1/2)\|f\|_{\infty }\ls c\sqrt{n}\|f\|_{\infty }<+\infty ,
\end{align*}
and the proof is complete.\end{proof}

\smallskip

In the case of the uniform measure $\mu=\mu_K$ on a centered convex body $K$ of volume $1$ in ${\mathbb R}^n$ we see that
$$ \int_K\big(\Lambda_{\mu_K}^{\ast }(x)/2n\big)^pdx\ls (c_1p)^p\int_Ke^{\frac{\Lambda_{\mu_K}^{\ast }(x)}{2n}}dx
\ls (c_2p)^p\sqrt{n},$$
where $c_1,c_2>0$ are absolute constants. This gives the following estimate for the moments of $\Lambda_{\mu_K}^{\ast}$:
$$\|\Lambda_{\mu_K}^{\ast }\|_{L^p(\mu_K)}\ls cpn^{1+\frac{1}{2p}}$$
for all $p\gr 1$. However, essentially repeating the argument that we used for Theorem~$\ref{th:moments-kappa}$ we may
obtain sharp estimates in the most interesting case $p=1$ or $2$. We need the next lemma.

\begin{lemma}\label{lem:digamma}Let $H_n=1+\frac{1}{2}+\cdots+\frac{1}{n}$. Then,
$$\int_0^1 r^{n-1}\ln(1-r)\,dr =-\frac{1}{n}H_n$$
and
$$\int_0^1 r^{n-1}\ln^2(1-r)\,dr=\frac{1}{n}H_n^2-\frac{1}{n}\sum_{k=1}^n\frac{1}{k^2}.$$
\end{lemma}

\begin{proof}We consider the beta integral
$$B(x,y)=\int_0^1 r^{x-1}(1-r)^{y-1}\, dr$$
and differentiate it with respect to $y$. Then, the desired integrals are equal to
$$\frac{\partial}{\partial y}B(x,y)\Big|_{y=1}\qquad\hbox{and}\qquad \frac{\partial^2}{\partial^2 y}B(x,y)\Big|_{y=1}.$$
We have
$$\frac{\partial}{\partial y} B(x, y) = B(x, y) \left( \frac{\Gamma'(y)}{\Gamma(y)} - \frac{\Gamma'(x + y)}{\Gamma(x + y)} \right) = B(x, y) \big(\psi(y) - \psi(x + y)\big),$$
where $\psi(y)=\frac{\Gamma'(y)}{\Gamma(y)}$ is the digamma function. Moreover,
$$\frac{\partial^2}{\partial^2 y} B(x, y)= B(x, y) \big((\psi(y) - \psi(x + y))^2-(\psi^{\prime}(y) - \psi^{\prime}(x + y))\big).$$
Recall (see e.g. \cite{Artin-book}) that
$$\psi(n+1)-\psi(1)=H_n:=\sum_{k=1}^n\frac{1}{k} \qquad\hbox{and}\qquad \psi^{\prime }(n)=\sum_{k=n}^{\infty }\frac{1}{k^2}$$
for all $n\gr 1$. Therefore,
\begin{align*}
\int_0^1 r^{n-1}\ln(1-r)\,dr =B(n, 1) \big(\psi(1) - \psi(n + 1)\big)&=-\frac{1}{n}H_n,
\end{align*}
and
\begin{align*}
\int_0^1 r^{n-1}\ln^2(1-r)\,dr=B(n, 1) \big((\psi(1) - \psi(n + 1))^2-(\psi'(1) - \psi'(n + 1)\big)&=\frac{1}{n}H_n^2-\frac{1}{n}\sum_{k=1}^n\frac{1}{k^2},
\end{align*}
as claimed.
\end{proof}

\begin{theorem}\label{th:moments-kappa-p=1}Let $K$ be a centered convex body of volume $1$ in ${\mathbb R}^n$, $n\gr 2$. Let $\kappa\in (0,1/n]$ and let
$\mu$ be a centered $\kappa$-concave probability measure with ${\rm supp}(\mu)=K$. Then,
$${\mathbb E}_{\mu }(\Lambda_{\mu}^{\ast })\ls \big({\mathbb E}_{\mu}[(\Lambda_{\mu}^{\ast })^2]\big)^{1/2}\ls \frac{c\ln n}{\kappa }\|f\|_{\infty}^{1/2},$$
where $c>0$ is an absolute constant and $f$ is the density of $\mu $.
\end{theorem}

\begin{proof}Following the proof of Theorem~\ref{th:moments-kappa} we write
$$\int_K\big(\Lambda_{\mu}^{\ast }(x)\big)^2\,d\mu(x)\ls \int_K\ln^2\left(\frac{e^2}{\kappa }\frac{1}{(1-\|x\|_K)^{1/\kappa }}\right)d\mu(x).$$
If $f$ is the density of $\mu $ on $K$ and $\nu_K$ is the cone measure of $K$, using the inequality $\ln^2(ab)\ls 2(\ln^2a+\ln^2b)$ where $a,b>0$,
we may write
\begin{align*}
&\frac{1}{2}\int_K\ln^2\left(\frac{e^2}{\kappa }\frac{1}{(1-\|x\|_K)^{1/\kappa }}\right)d\mu(x) -\ln^2\left(\frac{e^2}{\kappa}\right)\\
&\hspace*{2cm}\ls\int_{{\mathbb R}^n}f(x)\ln^2\left(\frac{1}{(1-\|x\|_K)^{1/\kappa }}\right){\mathbf 1}_K(x)\,dx\\
&\hspace*{2cm}= n|K|\int_0^{\infty }r^{n-1}\int_{\partial (K)}f(ry)\ln^2\left(\frac{1}{(1-\|ry\|_K)^{1/\kappa }}\right)
{\mathbf 1}_K(ry)\,d\nu_K(y)\,dr\\
&\hspace*{2cm}=\frac{n}{\kappa^2}\int_0^1r^{n-1}\ln^2(1-r)\int_{\partial (K)}f(ry)\,d\nu_K(y)\,dr\\
&\hspace*{2cm}\ls \frac{n}{\kappa^2}\|f\|_{\infty }\int_0^1r^{n-1}\ln^2(1-r)\,dr.
\end{align*}
Since $1\ls \int_Kf(x)\,dx\ls \|f\|_{\infty}$, using also Lemma~\ref{lem:digamma} we get
\begin{align*}\int_K\big(\Lambda_{\mu}^{\ast }(x)\big)^2\,d\mu(x) &\ls \frac{2n}{\kappa^2}\left(\frac{1}{n}H_n^2-\frac{1}{n}\sum_{k=1}^n\frac{1}{k^2}\right)\|f\|_{\infty} +2\ln^2\left(\frac{e^2}{\kappa}\right)\\
&\ls \left(\frac{2H_n^2}{\kappa^2}+2\ln^2(e^2/\kappa)\right)\|f\|_{\infty}\ls \frac{c_1\ln^2n}{\kappa^2}\|f\|_{\infty },
\end{align*}
where $c_1>0$ is an absolute constant. This completes the proof.\end{proof}

\smallskip

Our next result concerns the one-dimensional case. Let $\mu $ be a centered probability measure on ${\mathbb R}$
which is absolutely continuous with respect to Lebesgue measure
and consider a random variable $X$, on some probability space $(\varOmega,\mathcal{F},P),$ with distribution
$\mu$, i.e., $\mu(B):=P(X\in B),\ B\in\mathcal{B}(\mathbb{R})$. We define
$$\alpha_+  =\alpha_+  (\mu ):=\sup\set{x\in\mathbb{R}\colon
\mu([x,\infty))>0})\quad\hbox{and}\quad \alpha_-  =\alpha_-  (\mu ):=\sup\set{x\in\mathbb{R}\colon
\mu((-\infty,-x]))>0}).$$ Thus, $-\alpha_-,\alpha_+$ are the endpoints of the
support of $\mu $. Note that we may have $\alpha_{\pm} =+\infty$. We define $I_{\mu}=(-\alpha_- ,\alpha_+ )$. Recall that
$$\Lambda_{\mu}^{\ast }(x):=\sup\{tx-\Lambda_{\mu} (t)\colon t\in\mathbb{R}\},\qquad
x\in\mathbb{R}.$$
In fact, since $tx-\Lambda_{\mu}(t)<0$ for $t<0$ when $x\in [0,\alpha_+ )$, we have that
$\Lambda_{\mu}^{\ast }(x)=\sup\{tx-\Lambda_{\mu} (t)\colon t\gr 0\}$ in this case, and similarly
$\Lambda_{\mu}^{\ast }(x):=\sup\{tx-\Lambda_{\mu} (t)\colon t\ls 0\}$ when $x\in (-\alpha_- ,0]$.
One can also check that $\Lambda_{\mu}^{\ast }(\alpha_{\pm} )=+\infty$. See \cite[Lemma~2.8]{Gatzouras-Giannopoulos-2009}
for the case $\alpha_{\pm}<+\infty $. In the case $\alpha_{\pm}=\pm\infty$, the convexity and monotonicity properties
of $\Lambda_{\mu}^{\ast}$ imply again that $\lim\limits_{t\to\pm\infty}\Lambda_{\mu}^{\ast }(t)=+\infty$.

\begin{proposition}\label{prop:expectation-1}Let $\mu $ be a centered probability measure on ${\mathbb R}$
which is absolutely continuous with respect to Lebesgue measure. Then,
$$\int_{I_{\mu}}e^{\Lambda_{\mu}^{\ast }(x)/2}d\mu(x)\ls 4,$$
where $I_{\mu}={\rm supp}(\mu)$. In particular, for all $p\gr 1$ we have that
$$\int_{I_{\mu}}(\Lambda_{\mu}^{\ast }(x))^p\,d\mu(x)<+\infty .$$
\end{proposition}

\begin{proof}Let $F(x)=\mu (-\infty ,x]$. For any $x\in [0,\alpha_+ )$ and $t\gr 0$ we have
$$\min\{F(x),1-F(x)\}=\varphi_{\mu }(x)\ls e^{-\Lambda_{\mu}^{\ast }(x)}.$$
It follows that
\begin{align}\label{eq:two-summands}\int_{I_{\mu}}e^{\Lambda_{\mu}^{\ast }(x)/2}d\mu(x) &\ls
\int_{I_{\mu}}\frac{1}{\sqrt{\min\{F(x),1-F(x)\}}}f(x)\,dx\\
\nonumber &\ls \int_{I_{\mu}}\frac{1}{\sqrt{F(x)}}f(x)\,dx
+\int_{I_{\mu}}\frac{1}{\sqrt{1-F(x)}}f(x)\,dx.\end{align}
Write $f$ for the density of $\mu $ with respect to Lebesgue measure. Then, $(1-F)^{\prime }(x)=-f(x)$, which implies that
$$\int_0^{\alpha_+ }\frac{1}{\sqrt{1-F(x)}}f(x)\,dx\ls -\int_0^{\alpha_+ }\frac{1}{\sqrt{1-F(x)}}(1-F)^{\prime }(x)dx\\
=-2\sqrt{1-F(x)}\Big|_0^{\alpha_+ }=2\sqrt{1-F(0)}$$
since $F(\alpha_+ )=1$. In the same way we check that
$$\int_{-\alpha_-}^0\frac{1}{\sqrt{1-F(x)}}f(x)\,dx\ls -\int_{-\alpha_-}^0\frac{1}{\sqrt{1-F(x)}}(1-F)^{\prime }(x)dx\\
=-2\sqrt{1-F(x)}\Big|_{-\alpha_-}^0=2-2\sqrt{1-F(0)}.$$
This shows that
$$\int_{I_{\mu}}\frac{1}{\sqrt{1-F(x)}}f(x)\,dx\ls 2.$$
In a similar way we obtain the same upper bound for the second summand in \eqref{eq:two-summands}
and the result follows. \end{proof}

Proposition~\ref{prop:expectation-1} can be extended to products. Let $\mu_i$, $1\ls i\ls n$ be centered probability measures
on ${\mathbb R}$, all of them absolutely continuous with respect to Lebesgue measure. If $\overline{\mu}=\mu_1\otimes\cdots\otimes\mu_n$
then $I_{\overline{\mu}}=\prod_{i=1}^nI_{\mu_i}$ and we can easily check that
$$\Lambda_{\overline{\mu}}^{\ast }(x)=\sum_{i=1}^n\Lambda_{\mu_i}^{\ast }(x_i)$$
for all $x=(x_1,\ldots ,x_n)\in I_{\overline{\mu}}$, which implies that
$$\int_{I_{\overline{\mu}}}e^{\Lambda_{\overline{\mu}}^{\ast }(x)/2}d\overline{\mu}(x)
=\prod_{i=1}^n\left(\int_{I_{\mu_i}}e^{\Lambda_{\mu_i}^{\ast }(x_i)/2}d\mu_i(x_i)\right)\ls 4^n.$$
In particular, for all $p\gr 1$ we have that
$$\int_{I_{\overline{\mu}}}(\Lambda_{\overline{\mu}}^{\ast }(x))^p\,d\overline{\mu}(x)<+\infty .$$

We close this section with one more case where we can establish that $\Lambda_{\mu}^{\ast }$ has finite moments of all orders.
We consider an arbitrary centered log-concave probability measure on ${\mathbb R}^n$ but we have to impose
some conditions on the growth of its one-sided $L_t$-centroid bodies $Z_t^+(\mu )$. Recall that for every $t\gr 1$,
the one-sided $L_t$-centroid body $Z_t^+(\mu )$ of $\mu$ is the convex body with support function
\begin{equation*}h_{Z_t^+(\mu )}(y)=\left ( 2\int_{{\mathbb R}^n}\langle
x,y\rangle_+^tf_{\mu }(x)dx\right )^{1/t},\end{equation*} where $a_+=\max
\{a,0\}$. If $\mu $ is isotropic then $Z_2^+(\mu )\supseteq cB_2^n$ for an absolute constant $c>0$.
One can also check that if $1\ls t<s$ then
\begin{equation*}\left(\frac{2}{e}\right)^{\frac{1}{t}-\frac{1}{s}}Z_t^+(\mu )\subseteq Z_s^+(\mu )
\subseteq c_1\left(\frac{2e-2}{e}\right)^{\frac{1}{t}-\frac{1}{s}}\frac{s}{t}Z_t^+(\mu ).\end{equation*}
For a proof of these claims see \cite{Guedon-EMilman-2011}.
The condition we need is that the family of the one-sided $L_t$-centroid bodies
grows with some mild rate as $t\to\infty $ (note that the assumption in the next proposition
can be satisfied only for log-concave probability measures $\mu $ with support ${\rm supp}(\mu)={\mathbb R}^n$).

\begin{proposition}\label{prop:rate}Let $\mu $ be a centered log-concave probability measure on ${\mathbb R}^n$. Assume that
there exists an increasing function $g:[1,\infty )\to [1,\infty )$ with $\lim_{t\to\infty }g(t)/\ln (t+1)=+\infty $
such that $Z_t^+(\mu )\supseteq g(t)Z_2^+(\mu )$ for all $t\gr 2$. Then,
$$\int_{{\mathbb R}^n}|\Lambda_{\mu}^{\ast }(x)|^pd\mu(x)<+\infty $$
for every $p\gr 1$.
\end{proposition}

\begin{proof}We use the following fact, proved in \cite[Lemma~4.3]{BGP}: If $t\gr 1$ then for every $x\in \tfrac{1}{2}Z_t^+(\mu )$ we
have \begin{equation*}\varphi_{\mu }(x)\gr e^{-c_1t},\end{equation*}
where $c_1>1$ is an absolute constant. Since $\Lambda_{\mu}^{\ast}(x)\ls\ln\frac{1}{\varphi_{\mu}(x)}$, this shows
that $\Lambda_{\mu}^{\ast }(x)\ls c_1t$ for all $x\in \tfrac{1}{2}Z_t^+(\mu )$. In other words,
\begin{equation}\label{eq:Zt}\frac{1}{2}Z_{t/c_1}^+(\mu)\subseteq B_t(\mu),\qquad t\gr c_1.\end{equation}
Since $\lim_{t\to\infty }g(t)=+\infty $, there exists $t_0\gr c_1$
such that $\mu \left(\frac{g(t_0/c_1)}{2}Z_2^+(\mu)\right)\gr 2/3$. From Borell's lemma \cite[Lemma~2.4.5]{BGVV-book}
we know that, for all $t\gr t_0$,
$$1-\mu \left(\frac{g(t/c_1)}{2}Z_2^+(\mu )\right)\ls e^{-c_2g(t/c_1)/g(t_0/c_1)},$$
where $c_2>0$ is an absolute constant. We write
$$\int_{{\mathbb R}^n}|\Lambda_{\mu}^{\ast }(x)|^pd\mu(x)=\int_0^{\infty }pt^{p-1}\mu (\{x:\Lambda_{\mu }^{\ast }(x)\gr t\})\,dt
=p\int_0^{\infty }t^{p-1}(1-\mu(B_t(\mu)))\,dt.$$
From \eqref{eq:Zt} it follows that
$$1-\mu(B_t(\mu))\ls 1-\mu\left(\frac{1}{2}Z_{t/c_1}^+(\mu)\right)\ls 1-\mu\left(\frac{g(t/c_1)}{2}Z_2^+(\mu)\right)\ls e^{-c_2g(t/c_1)/g(t_0/c_1)}$$
for all $t\gr t_0$. Since $\lim_{t\to\infty }g(t)/\ln (t+1)=+\infty $, there exists $t_p\gr t_0$ such that
$$(p-1)\ln (t)\ls \frac{c_2}{2g(t_0/c_1)}g(t/c_1)$$
for all $t\gr t_p$. Assume that $p>2$. Then, from the previous observations we get
\begin{align*}p\int_{t_p}^{\infty }t^{p-1}(1-\mu(B_t(\mu)))\,dt &\ls p\int_{t_p}^{\infty }t^{p-1}\left(1-\mu\left(\tfrac{g(t/c_1)}{2}Z_2^+(\mu)\right)\right)\,dt\\
&\ls p\int_{t_p}^{\infty }t^{p-1}t^{-2(p-1)}\,dt= p\int_{t_p}^{\infty }t^{-(p-1)}\,dt <\infty .
\end{align*}
This proves the result for $p>2$ and then from H\"{o}lder's inequality it is clear that the assertion of the proposition is
also true for all $p\gr 1$. \end{proof}

\begin{note*}\rm It is not hard to construct examples of log-concave probability measures, even on the real line, for which
${\rm supp}(\mu)={\mathbb R}^n$ but the assumption of Proposition~\ref{prop:rate} is not satisfied. Consider for example
a measure $\mu$ on ${\mathbb R}$ with density $f(x)=c\cdot\exp(-p)$ where $p$ is an even convex function rapidly increasing
to infinity, e.g. $p(t)=e^{t^2}$.

However, this does not exclude the possibility that for every centered log-concave probability measure $\mu $ on ${\mathbb R}^n$
the function $\Lambda_{\mu}^{\ast }$ has finite second or higher moments.
\end{note*}

\section{Threshold for the measure: the approach and examples}\label{section-5}

For any log-concave probability measure $\mu $ on ${\mathbb R}^n$ we define the parameter
\begin{equation}\label{eq:beta}\beta(\mu)=\frac{{\rm Var}_{\mu }(\Lambda_{\mu}^{\ast })}{({\mathbb E}_{\mu }(\Lambda_{\mu }^{\ast }))^2}\end{equation}
provided that $$\|\Lambda_{\mu }^{\ast }\|_{L^2(\mu )}=({\mathbb E}(\Lambda_{\mu }^{\ast })^2)^{1/2}<\infty .$$
One of the main results in \cite{BGP} states that if $\mu $ is a log-concave probability measure on ${\mathbb R}^n$ then
$$\int_{\mathbb{R}^n}\varphi_{\mu }(x)\,d\mu(x) \ls \exp\left(-cn/L_{\mu }^2\right),$$
where $c>0$ is an absolute constant. In fact, the proof of this estimate starts with Lemma~\ref{lem:upper-bt} and follows
from the next stronger result: If $n\gr n_0$ then
$$\int_{{\mathbb R}^n}\exp (-\Lambda_{\mu }^{\ast }(x))\,d\mu(x) \ls \exp\left(-cn/L_{\mu}^2\right)$$
where $L_{\mu }$ is the isotropic constant of $\mu $ and $c>0$, $n_0\in {\mathbb N}$ are absolute constants.
Then, Jensen's inequality implies that
$$e^{-{\mathbb E}_{\mu }(\Lambda_{\mu }^{\ast })}\ls \int_{{\mathbb R}^n}\exp (-\Lambda_{\mu }^{\ast }(x))\,d\mu(x)
\ls \exp\left(-cn/L_{\mu}^2\right).$$
We will need this lower bound for ${\mathbb E}_{\mu }(\Lambda_{\mu }^{\ast })$.

\begin{lemma}\label{lem:lower-mean-L*}
Let $\mu$ be a log-concave probability measure on ${\mathbb R}^n$, $n\gr n_0$. Then,
$${\mathbb E}(\Lambda_{\mu }^{\ast })\gr cn/L_{\mu}^2,$$
where $L_{\mu }$ is the isotropic constant of $\mu $ and $c>0$, $n_0\in {\mathbb N}$ are absolute constants.
\end{lemma}

We will also need a number of observations in the case $\mu=\mu_K$ where $K$ is a centered convex body of volume $1$ in ${\mathbb R}^n$.
The next lemma provides a lower bound for ${\rm Var}(\Lambda_{\mu_K}^{\ast })$.

\begin{lemma}\label{lem:body-lower-bound-for-variance}Let $K$ be a centered convex body of volume $1$ in ${\mathbb R}^n$.
Then,
$${\rm Var}(\Lambda_{\mu_K}^{\ast })\gr c/L_{\mu_K}^4,$$
where $c>0$ is an absolute constant.
\end{lemma}

\begin{proof}Borell has proved in \cite[Theorem~1]{Borell-1973} that if $T$ is a convex body in ${\mathbb R}^n$ and $f$ is a non-negative,
bounded and convex function on $T$, not identically zero and with $\min(f)=0$, then the function
$$\Phi_f(p)=\ln\big[(n+p)\|f\|_p^p\big ]$$
is convex on $[0,\infty )$. Consider a centered convex body $K$ of volume $1$ in ${\mathbb R}^n$. Applying Borell's
theorem for the function $\Lambda_{\mu_K}^{\ast }$ on $rK$, $r\in (0,1)$ and the triple $p=0,1$ and $2$, and finally
letting $r\to 1^-$, we see that
$$(n+1)^2\|\Lambda_{\mu_K}^{\ast }\|_{L^1(\mu_K)}^2\ls n(n+2)\|\Lambda_{\mu_K}^{\ast }\|_{L^2(\mu_K)}^2,$$
which implies that
$${\rm Var}(\Lambda_{\mu_K}^{\ast })\gr \frac{1}{n(n+2)}\|\Lambda_{\mu_K}^{\ast }\|_{L^1(\mu_K)}^2.$$
Then, taking into account Lemma~\ref{lem:lower-mean-L*} we obtain the result.
\end{proof}

Recall the definition of $\omega_{\mu_K}=\ln(1/\varphi_{\mu_K})$ in \eqref{eq:omega} and consider the parameter
\begin{equation}\label{eq:tau}\tau (\mu_K)=\frac{{\rm Var}_{\mu_K }(\omega_{\mu_K})}{({\mathbb E}_{\mu_K }(\omega_{\mu_K}))^2}.
\end{equation}
The next lemma shows that we can estimate $\beta(\mu_K)$ if we can compute $\tau(\mu_K)$.

\begin{lemma}\label{lem:body-beta-omega}Let $K$ be a centered convex body of volume $1$ in ${\mathbb R}^n$.
Then,
$$\beta(\mu_K)=\left(\tau(\mu_K)+O(L_{\mu_K}^2/\sqrt{n})\right)\left(1+O(L_{\mu_K}^2/\sqrt{n})\right).$$
\end{lemma}

\begin{proof}From Corollary~\ref{cor:L*-omega} we know that if $K$ is a centered convex body of volume $1$ in ${\mathbb R}^n$
then for every $x\in {\rm int}(K)$ we have that $\omega_{\mu_K} (x)-5\sqrt{n}\ls \Lambda_{\mu_K}^{\ast }(x)\ls \omega_{\mu_K}(x)$.
Writing $\Lambda_{\mu_K}^{\ast}=\omega_{\mu_K}+h$ where $\|h\|_{\infty }\ls 5\sqrt{n}$
we easily see that
$${\rm Var}_{\mu_K }(\Lambda_{\mu_K}^{\ast })= {\rm Var}_{\mu_K }(\omega_{\mu_K})+O(\sqrt{n}{\mathbb E}_{\mu_K }(\Lambda_{\mu_K }^{\ast }))$$
where $X=O(Y)$ means that $|X|\ls cY$ for an absolute constant $c>0$. Lemma~\ref{lem:lower-mean-L*} and the fact that
${\mathbb E}_{\mu_K}(\omega_{\mu_K})={\mathbb E}_{\mu_K}(\Lambda_{\mu_K }^{\ast })+O(\sqrt{n})$ imply that
$$\frac{{\mathbb E}_{\mu_K}(\omega_{\mu_K})}{{\mathbb E}_{\mu_K}(\Lambda_{\mu_K }^{\ast })}=1+O(L_{\mu_K}^2/\sqrt{n}).$$
Taking also into account the fact that $L_{\mu_K}^2/\sqrt{n}=O((\ln n)^8/\sqrt{n})=o(1)$ we get
$${\mathbb E}_{\mu_K}(\omega_{\mu_K})\approx {\mathbb E}_{\mu_K}(\Lambda_{\mu_K }^{\ast }).$$
Combining the above we see that
\begin{align*}
\beta(\mu_K) &= \frac{{\rm Var}_{\mu_K }(\Lambda_{\mu_K}^{\ast })}{({\mathbb E}_{\mu_K}(\Lambda_{\mu_K }^{\ast }))^2}
=\frac{{\rm Var}_{\mu_K }(\omega_{\mu_K})+O(\sqrt{n}{\mathbb E}_{\mu_K }(\Lambda_{\mu_K }^{\ast }))}{({\mathbb E}_{\mu_K}(\omega_{\mu_K}))^2}
\left(\frac{{\mathbb E}_{\mu_K}(\omega_{\mu_K})}{{\mathbb E}_{\mu_K}(\Lambda_{\mu_K }^{\ast })}\right)^2\\
&= \left(\frac{{\rm Var}_{\mu_K }(\omega_{\mu_K})}{({\mathbb E}_{\mu_K}(\omega_{\mu_K }))^2}+O\left(L_{\mu_K}^2/\sqrt{n}\right)\right)\left (1+O\left(L_{\mu_K}^2/\sqrt{n}\right)\right)^2\\
&=\left(\tau(\mu_K)+O\left(L_{\mu_K}^2/\sqrt{n}\right)\right)\left (1+O\left(L_{\mu_K}^2/\sqrt{n}\right)\right),
\end{align*}
as claimed.
\end{proof}

Recall that $B_t(\mu)=\{x\in{\mathbb R}^n:\Lambda^{\ast}_{\mu}(x)\ls t\}$, where $\Lambda^{\ast }_{\mu
}$ is the Cramer transform of $\mu $. A version of the next lemma appeared
originally in \cite{DFM}.

\begin{lemma}\label{lem:upper-1}Let $t>0$. For every $N>n$ we have
$${\mathbb E}_{\mu^N}(\mu (K_N))\ls \mu (B_t(\mu ))+ N\exp (-t).$$
\end{lemma}

We use the lemma in the following way. Let $m:=\mathbb{E}_{\mu}(\Lambda_{\mu}^{*})$. Then, for all $\epsilon\in (0,1)$,
from Chebyshev's inequality we have that
\begin{align*}
\mu(\{\Lambda_{\mu}^{*}\ls m -\epsilon m  \})&\ls \mu(\{|\Lambda_{\mu}^{*}-m |\gr \epsilon m \})
	\ls\frac{\mathbb{E}_{\mu}|\Lambda_{\mu}^{*}-m|^2}{\epsilon^2m^2}=\frac{\beta(\mu)}{\epsilon^2}.
\end{align*}
Equivalently,
$$\mu (B_{(1-\epsilon )m }(\mu))\ls \frac{\beta(\mu)}{\epsilon^2}.$$
Let $\delta\in (\beta(\mu),1)$. We distinguish two cases:

\smallskip

(i) If $\beta (\mu )<1/8$ and $8\beta(\mu )<\delta <1$ then, choosing $\epsilon =\sqrt{2\beta(\mu)/\delta }$ we have that
$$\mu (B_{(1-\epsilon )m }(\mu))\ls \frac{\delta }{2}.$$
Then, from Lemma~\ref{lem:upper-1} we see that
\begin{align*}\sup\{ {\mathbb E}_{\mu^N}(\mu (K_N)):N\ls e^{(1-2\epsilon )m} \}
&\ls \mu (B_{(1-\epsilon)m }(\mu))+e^{(1-2\epsilon )m }e^{-(1-\epsilon)m }\\
&\ls \frac{\delta }{2}+e^{-\epsilon m}\ls\delta ,
\end{align*}
provided that $\epsilon m\gr \ln (2/\delta )$. Since $m\gr c_1n/L_{\mu}^2$, this condition is satisfied if
$n/L_{\mu}^2\gr c_2\ln (2/\delta )\sqrt{\delta /\beta(\mu)}$. By the choice of $\epsilon $ we conclude that
$$\varrho_1(\mu ,\delta )\gr \left(1-\sqrt{8\beta(\mu)/\delta }\right)\frac{\mathbb{E}_{\mu}(\Lambda_{\mu}^{*})}{n}.$$

(ii) If $1/8\ls \beta (\mu )<1$  and $\beta(\mu )<\delta <1$ then, choosing $\epsilon =\sqrt{\frac{2\beta(\mu)}{\beta(\mu)+\delta }}$ we have that
$$\mu (B_{(1-\epsilon )m }(\mu))\ls \frac{\beta(\mu)+\delta }{2}.$$
Then, exactly as in (i), we see that
$$\sup\{ {\mathbb E}_{\mu^N}(\mu (K_N)):N\ls e^{(1-\sqrt{\epsilon} )m} \}
\ls \frac{\beta(\mu)+\delta }{2}+e^{(1-\sqrt{\epsilon} )m }e^{-(1-\epsilon)m }\ls \frac{\beta (\mu)+\delta }{2}+e^{-(\sqrt{\epsilon}-\epsilon) m}\ls\delta ,$$
provided that
\begin{equation}\label{eq:condition}(\sqrt{\epsilon}-\epsilon ) m\gr \ln \left(\frac{2}{\delta-\beta(\mu)}\right ).\end{equation}
Note that $1>\epsilon\gr\sqrt{\beta (\mu)}\gr\frac{1}{2\sqrt{2}}$, and hence
$$\sqrt{\epsilon}-\epsilon =\frac{\sqrt{\epsilon}}{(1+\sqrt{\epsilon })(1+\epsilon )}(1-\epsilon^2)\gr c_1^{\prime}(1-\epsilon^2)=c_1^{\prime}\frac{\delta -\beta(\mu)}{\delta +\beta(\mu)}\gr c_2^{\prime}(\delta -\beta(\mu)),$$
where $c_i^{\prime}>0$ are absolute constants. Since $m\gr c_1n/L_{\mu}^2$,  we see that if
$n/L_{\mu}^2\gr \frac{c_2}{\delta -\beta(\mu)}\ln \left(\frac{2}{\delta-\beta(\mu)}\right )$ then \eqref{eq:condition} is satisfied.
Therefore, we conclude that
$$\varrho_1(\mu ,\delta )\gr \left(1-\sqrt[4]{\frac{2\beta(\mu)}{\beta(\mu)+\delta}}\right)\frac{\mathbb{E}_{\mu}(\Lambda_{\mu}^{*})}{n}.$$
We summarize the above in the next theorem.

\begin{theorem}\label{th:upper-delta}Let $\mu$ be a log-concave probability measure on ${\mathbb R}^n$.
\begin{enumerate}
\item[{\rm (i)}] Let $\beta (\mu )<1/8$ and $8\beta(\mu )<\delta <1$. If $n/L_{\mu}^2\gr c_2\ln (2/\delta )\sqrt{\delta /\beta(\mu)}$ then
$$\varrho_1(\mu ,\delta )\gr \left(1-\sqrt{8\beta(\mu)/\delta }\right)\frac{\mathbb{E}_{\mu}(\Lambda_{\mu}^{*})}{n}.$$
\item[{\rm (ii)}] Let $1/8\ls\beta (\mu )<1$ and $\beta(\mu )<\delta <1$. If $n/L_{\mu}^2\gr \frac{c_2}{\delta -\beta(\mu)}\ln \left(\frac{2}{\delta-\beta(\mu)}\right )$ then
$$\varrho_1(\mu ,\delta )\gr \left(1-\sqrt[4]{\frac{2\beta(\mu)}{\beta(\mu)+\delta}}\right)\frac{\mathbb{E}_{\mu}(\Lambda_{\mu}^{*})}{n}.$$
\end{enumerate}
\end{theorem}

\begin{remark}\label{rem:pv}\rm Paouris and Valettas have proved in \cite[Theorem~5.6]{Paouris-Valettas-2019} that if $\mu$ is
a log-concave probability measure on ${\mathbb R}^n$ and $p$ is a convex function on ${\mathbb R}^n$ then
\begin{equation}\label{eq:pv-1}\mu\left(\left\{x:p(x)<M(p)-t\|p -M(p)\|_{L^1(\mu)}\right\}\right)
\ls \frac{1}{2}\exp(-t/16)\end{equation}
for all $t>0$, where $M(p)$ is a median of $p$ with respect to $\mu $. Consider the parameter
$$\tilde{\beta}(\mu)=\frac{\|\Lambda_{\mu}^{\ast }-{\mathbb E}(\Lambda_{\mu}^{\ast })\|_{L^2(\mu )}}{M(\Lambda_{\mu}^{\ast })}.$$
Recall that $\Lambda_{\mu}^{\ast }$ is convex and set $M=M(\Lambda_{\mu}^{\ast })$. Since
$$\|\Lambda_{\mu}^{\ast }-M(\Lambda_{\mu}^{\ast })\|_{L^1(\mu )}\ls \|\Lambda_{\mu}^{\ast }-{\mathbb E}(\Lambda_{\mu}^{\ast })\|_{L^1(\mu )}
\ls \|\Lambda_{\mu}^{\ast }-{\mathbb E}(\Lambda_{\mu}^{\ast })\|_{L^2(\mu )},$$
from \eqref{eq:pv-1} we see that
\begin{equation}\label{eq:pv-2}\mu\left(\left\{x:\Lambda_{\mu}^{\ast }(x)<(1-t\tilde{\beta}(\mu))M\right\}\right)
\ls \frac{1}{2}\exp(-t/16)\end{equation}
for every $t>0$. This shows that $$\mu (B_{(1-t\tilde{\beta}(\mu))M}(\mu))\ls \frac{\delta }{4}$$
if $t\gr 16\ln (2/\delta)$. Then, from Lemma~\ref{lem:upper-1} we see that
\begin{align*}\sup\{ {\mathbb E}_{\mu^N}(\mu (K_N)):N\ls e^{(1-2t\tilde{\beta}(\mu))M} \}
&\ls \mu (B_{(1-t\tilde{\beta}(\mu))M}(\mu))+e^{(1-2t\tilde{\beta}(\mu))M}e^{-(1-t\tilde{\beta}(\mu))M}\\
&= \mu (B_{(1-t\tilde{\beta}(\mu))M}(\mu))+e^{-t\tilde{\beta}(\mu)M}\ls \frac{\delta }{4}+e^{-t\tilde{\beta}(\mu)M}\ls\delta ,
\end{align*}
provided that $t\tilde{\beta}(\mu)M\gr \ln (2/\delta )$. Now, we restrict our attention to the case where
$\mu=\mu_K$ is the uniform measure on a centered convex body $K$ of volume $1$ in ${\mathbb R}^n$. Then,
Lemma~\ref{lem:body-lower-bound-for-variance} shows that
$$\tilde{\beta}(\mu_K)M(\Lambda_{\mu_K}^{\ast })=\|\Lambda_{\mu_K}^{\ast }-{\mathbb E}(\Lambda_{\mu_K}^{\ast })\|_{L^2(\mu )}\gr c_1/L_{\mu_K}^2$$
where $c_1>0$ is an absolute constant. Choosing $t=c_2L_{\mu_K}^2\ln (2/\delta )$ we get
\begin{equation}\label{eq:muK-M}\varrho_1(\mu_K ,\delta )\gr \left(1-c_3L_{\mu_K}^2\ln (2/\delta)\tilde{\beta}(\mu_K)\right)\frac{M(\Lambda_{\mu_K}^{\ast })}{n}.\end{equation}
A natural question is to examine how close ${\mathbb E}(\Lambda_{\mu}^{\ast })$ and $M(\Lambda_{\mu }^{\ast })$ are.
This would allow us to compare~\eqref{eq:muK-M} with Theorem~\ref{th:upper-delta} at least in the case of the uniform
measure on a convex body. From \cite[Lemma~2.4.10]{BGVV-book} we know that
$$\frac{1}{2}\|\Lambda_{\mu}^{\ast }-{\mathbb E}(\Lambda_{\mu}^{\ast })\|_{L^2(\mu )}\ls \|\Lambda_{\mu}^{\ast }-M(\Lambda_{\mu}^{\ast })\|_{L^2(\mu )}
\ls 3\|\Lambda_{\mu}^{\ast }-{\mathbb E}(\Lambda_{\mu}^{\ast })\|_{L^2(\mu )}$$
for any Borel probability measure $\mu$ on ${\mathbb R}^n$. Therefore, if we assume that $\beta (\mu_K)\ls \eta $ for some
small enough $\eta\in (0,1)$, we see that
$$M(\Lambda_{\mu}^{\ast })\gr \|\Lambda_{\mu}^{\ast }\|_2-3\sqrt{\eta }\|\Lambda_{\mu}^{\ast }\|_1\gr (1-3\sqrt{\eta }){\mathbb E}(\Lambda_{\mu}^{\ast }).$$
This gives a variant of Theorem~\ref{th:upper-delta} with a much better dependence on $\delta $.
\end{remark}

\begin{theorem}\label{th:variant-upper-delta}Let $K$ be a centered convex body of volume $1$ in ${\mathbb R}^n$
and let $\delta\in (0,1)$ and $\eta\in (0,1/9)$. If $\beta (\mu_K)\ls \eta$ and $c_3L_{\mu_K}^2\ln (2/\delta)\tilde{\beta}(\mu_K) <1$ then
$$\varrho_1(\mu ,\delta )\gr \left(1-c_3L_{\mu_K}^2\ln (2/\delta)\tilde{\beta}(\mu_K)\right)(1-3\sqrt{\eta })\frac{{\mathbb E}(\Lambda_{\mu_K}^{\ast })}{n}.$$
\end{theorem}

For the proof of the lower threshold we need a basic fact that plays a main role in the proof of all
the lower thresholds that have been obtained so far. It is stated in the form below in \cite[Lemma~3]{Chakraborti-Tkocz-Vritsiou-2021}.
For a proof see \cite{DFM} or \cite[Lemma~4.1]{Gatzouras-Giannopoulos-2009}.

\begin{lemma}\label{lem:inclusion}For every Borel subset $A$ of ${\mathbb R}^n$ we have that
$$1-\mu^N(K_N\supseteq A)\ls 2\binom{N}{n}\left (1-\inf_{x\in A}\varphi_{\mu }(x)\right)^{N-n}.$$
Therefore,
$${\mathbb E}\,[\mu (K_N)]\gr \mu (A)\left (1-2\binom{N}{n}\left (1-\inf_{x\in A}\varphi_{\mu }(x)\right)^{N-n}\right).$$
\end{lemma}

In order to apply Lemma~\ref{lem:inclusion} we note that if $m:=\mathbb{E}_{\mu}(\Lambda_{\mu}^{*})$ then as before, for all $\epsilon\in (0,1)$,
from Chebyshev's inequality we have that
\begin{equation*}
\mu(\{\Lambda_{\mu}^{*}\gr m +\epsilon m  \})\ls \mu(\{|\Lambda_{\mu}^{*}-m |\gr \epsilon m \})\ls\frac{\beta(\mu)}{\epsilon^2}.
\end{equation*}
If $\beta (\mu )<1/2$ and $2\beta(\mu )<\delta <1$ then, choosing $\epsilon =\sqrt{2\beta(\mu)/\delta }$ we have that
$$\mu (B_{(1+\epsilon )m }(\mu))\gr 1-\frac{\delta}{2}.$$
Therefore, we will have that
$$\varrho_2(\mu,\delta )\ls (1+2\epsilon)m/n$$
if our lower bound for $\inf_{x\in B_{(1+\epsilon )m}(\mu)}\varphi_{\mu }(x)$ gives
\begin{equation}\label{eq:tedious-1}2\binom{N}{n}\left (1-\inf_{x\in B_{(1+\epsilon )m}(\mu)}\varphi_{\mu }(x)\right)^{N-n}\ls\frac{\delta}{2}\end{equation}
for all $N\gr N_0:=\exp((1+2\epsilon )m)$.
Recall that in the case of the uniform measure on a centered convex body of volume $1$, Theorem~\ref{th:lower-bt-body} shows that
$$\inf_{x\in B_{(1+\epsilon)m}(\mu_K)}\varphi_{\mu_K}(x)\gr \frac{1}{10}\exp (-(1+\epsilon)m-2\sqrt{n}).$$
We require that $n$ and $m$ are large enough so that $1/2^n<\delta /2$ and
$2\sqrt{n}\ls\frac{\epsilon m}{2}$. Using also the fact that $\binom{N}{n}\ls e^{-1}\left(\frac{eN}{n}\right)^n$ we see that \eqref{eq:tedious-1} will be
satisfied if we also have
$$\left(\frac{2eN}{n}\right)^ne^{-\frac{N-n}{10}e^{-(1+3\epsilon/2)m}}<1.$$
Setting $x:=N/n$ we see that this last is equivalent to
$$e^{(1+3\epsilon/2)m}<\frac{x-1}{10\ln (2ex)}.$$
One can now check that if $N\gr \exp ((1+2\epsilon )m)$ then all the restrictions are satisfied if we assume
that $n/L_{\mu_K}^2\gr c_2\ln (2/\delta )\sqrt{\delta /\beta(\mu_K)}$. In this way we get the following.

\begin{theorem}\label{th:lower-delta}Let $\beta ,\delta>0$ with $2\beta <\delta <1$. If $K$ is a centered convex body of volume $1$ in ${\mathbb R}^n$ with $\beta (\mu_K )=\beta $ and $n/L_{\mu_K}^2\gr c_2\ln (2/\delta )\sqrt{\delta /\beta}$ then
$$\varrho_2(\mu_K ,\delta )\ls \left(1+\sqrt{8\beta/\delta }\right)\frac{\mathbb{E}_{\mu_K}(\Lambda_{\mu_K}^{*})}{n}.$$
\end{theorem}

An estimate analogous to the one in Theorem~\ref{th:upper-delta}\,(ii) is also possible but we shall not go through the details.
From the discussion in this section it is clear that our approach is able to provide good bounds for the threshold $\varrho(\mu,\delta)$
if the parameter $\beta(\mu)$ is small, especially if $\beta(\mu)=o_n(1)$ as the dimension increases. We illustrate this with a
number of examples.

\begin{example}[Uniform measure on the cube]\rm Let $\mu_{C_n}$ be the uniform measure on the unit cube $C_n=\left[-\tfrac{1}{2},\tfrac{1}{2}\right]^n$.
Since $\mu_{C_n}=\mu_{C_1}\otimes\cdots\otimes\mu_{C_1}$ we have
$${\rm Var}_{\mu_{C_n}}(\Lambda_{\mu_{C_n}}^{\ast})=n{\rm Var}_{\mu_{C_1}}(\Lambda_{\mu_{C_1}}^{\ast})\qquad\hbox{and}
\qquad {\mathbb E}_{\mu_{C_n} }(\Lambda_{\mu_{C_n} }^{\ast })=n{\mathbb E}_{\mu_{C_1} }(\Lambda_{\mu_{C_1} }^{\ast }).$$
Therefore,
$$\beta(\mu_{C_n})=\frac{{\rm Var}_{\mu_{C_n}}(\Lambda_{\mu_{C_n}}^{\ast})}{({\mathbb E}_{\mu_{C_n} }(\Lambda_{\mu_{C_n} }^{\ast }))^2}
=\frac{\beta(\mu_{C_1})}{n}\longrightarrow 0.$$
as $n\to\infty $. Then, Theorem~\ref{th:upper-delta} and Theorem~\ref{th:lower-delta} show that for any $\delta\in (0,1)$ there exists
$n_0(\delta )$ such that, for any $n\gr n_0$,
$$\varrho_1(\mu_{C_n},\delta)\gr\left(1-\sqrt{\frac{8\beta_{\mu_{C_n}}}{\delta}}\right)\frac{{\mathbb E}(\Lambda_{\mu_{C_n}}^{\ast })}{n}
\gr \left(1-\frac{c_1}{\sqrt{\delta n}}\right){\mathbb E}(\Lambda_{\mu_{C_1}}^{\ast })$$
and
$$\varrho_2(\mu_{C_n},\delta)\ls\left(1+\sqrt{\frac{8\beta_{\mu_{C_n}}}{\delta}}\right)\frac{{\mathbb E}(\Lambda_{\mu_{C_n}}^{\ast })}{n}
\ls \left(1+\frac{c_2}{\sqrt{\delta n}}\right){\mathbb E}(\Lambda_{\mu_{C_1}}^{\ast }),$$
which shows that
$$\varrho (\mu_{C_n},\delta )\ls \frac{c}{\sqrt{\delta n}},$$
where $c>0$ is an absolute constant. This estimate provides a sharp threshold for the measure of a random polytope $K_N$
with independent vertices uniformly distributed in $C_n$. It provides a direct proof of the result of Dyer, F\"{u}redi and McDiarmid
in \cite{DFM} with a stronger estimate for the ``width of the threshold".
\end{example}

\begin{example}[Gaussian measure]\rm
Let $\gamma_n$ denote the standard $n$-dimensional Gaussian measure with density $f_{\gamma_n}(x)=(2\pi )^{-n/2}e^{-|x|^2/2}$, $x\in {\mathbb R}^n$.
Note that $\gamma_n=\gamma_1\otimes\cdots\otimes\gamma_1$, and hence we may argue as in the previous example.
We may also use direct computation to see that
\begin{equation*}\Lambda_{\gamma_n}(\xi )=\log\Big(\int_{{\mathbb R}^n}e^{\langle\xi
,z\rangle }f_{\gamma_n}(z)dz\Big)=\frac{1}{2}|\xi|^2\end{equation*}
for all $\xi\in {\mathbb R}^n$ and
\begin{equation*}\Lambda_{\gamma_n }^{\ast }(x)
= \sup_{\xi\in {\mathbb R}^n} \left\{ \langle x, \xi\rangle - \Lambda_{\gamma_n}(\xi )\right\}=\frac{1}{2}|x|^2\end{equation*}
for all $x\in {\mathbb R}^n$. It follows that $$B_t(\gamma_n)=\{x\in {\mathbb R}^n:\Lambda_{\gamma_n}^{\ast }(x)\ls t\}=\{x\in {\mathbb R}^n:|x|\ls \sqrt{2t}\}
=\sqrt{2t}B_2^n.$$
Note that if $x\in\partial (B_t(\gamma_n))$ then
$$\varphi_{\gamma_n}(x)=\frac{1}{\sqrt{2\pi }}\int_{\sqrt{2t}}^{\infty }e^{-u^2/2}du\gr \frac{c}{\sqrt{t}}e^{-t}$$
for all $t\gr 1$ (see \cite[p. 17]{IMK} for a refined form of the lower bound that we use). By the standard concentration estimate for the Euclidean norm with respect to $\gamma_n$ (see \cite[Theorem~3.1.1]{Vershynin-book}),
we have $\|\;|x|-\sqrt{n}\;\|_{\psi_2}\ls C$, where $C>0$ is an absolute constant, or equivalently, for any $s>0$,
$$\gamma_n(\{x\in {\mathbb R}^n:|\;|x|-\sqrt{n}\;|\gr s\sqrt{n}\})\ls 2\exp (-cs^2n),$$
where $c>0$ is an absolute constant. This shows that
$$\max\{\gamma_n((1-s)\sqrt{n}B_2^n), 1-\gamma_n((1+s)\sqrt{n}B_2^n)\}\ls 2\exp (-cs^2n)$$
for every $s\in (0,1)$. Let $\epsilon\in (0,1/2)$. Applying Lemma~\ref{lem:upper-1} with $t=\left(1-\epsilon\right)n/2$ and
$N\ls \exp ((1/2-\epsilon)n)$ we see that
\begin{align*}{\mathbb E}_{\gamma_n^N}(\gamma_n (K_N)) &\ls \gamma_n (\sqrt{(1-\epsilon)n}B_2^n)+ \exp (-\epsilon n/2)\\
&\ls 2\exp (-c\epsilon^2n)+\exp (-\epsilon n/2),
\end{align*}
using the fact that $\sqrt{1-\epsilon}\ls 1-\epsilon /2$. It follows that, for any $\delta\in (0,1)$, if we choose $\epsilon =c_1\sqrt{\ln(4/\delta)}/\sqrt{n}$ we have
$$\sup\left\{{\mathbb E}_{\gamma_n^N}\big(\gamma_n(K_N)\big):N\ls e^{\left(\frac{1}{2}-\epsilon\right)n}\right\}\ls\delta ,$$
and hence $$\varrho_1(\gamma_n,\delta)\gr\frac{1}{2}-\frac{c_1\sqrt{\ln (4/\delta)}}{\sqrt{n}}.$$
Now, let $N\gr \exp ((1/2+\epsilon )n)$. Applying Lemma~\ref{lem:inclusion} with $A=B_t(\gamma_n)$ where $t=(1+\epsilon)n/2$,
we see that $\gamma_n (B_t(\gamma_n))=\gamma_n(\sqrt{(1+\epsilon )n}B_2^n)\gr 1-2\exp (-c\epsilon^2n)$, because $\sqrt{1+\epsilon}\gr 1+\epsilon /3$.
We also have
$$2\binom{N}{n}\left (1-\inf_{x\in B_t(\gamma_n)}\varphi_{\gamma_n }(x)\right)^{N-n}\ls \left(\frac{2eN}{n}\right)^n
\exp\left(-\frac{c(N-n)}{\sqrt{n}}e^{-(1+\epsilon)n/2}\right).$$
Let $\delta\in (0,1)$. We choose $\epsilon =c_2\sqrt{\ln(4/\delta)}/\sqrt{n}$
and insert these estimates into Lemma~\ref{lem:inclusion}. Arguing as in the proof of \eqref{eq:tedious-1} we see that if $n\gr n_0(\delta)$
then
$$\inf\left\{{\mathbb E}_{\gamma_n^N}\big(\gamma_n(K_N)\big):N\gr e^{\left(\frac{1}{2}+\epsilon\right)n}\right\}\gr 1-\delta ,$$
and hence
$$\varrho_2(\gamma_n,\delta)\ls \frac{1}{2}+\frac{c_2\sqrt{\ln (4/\delta)}}{\sqrt{n}}.$$
Combining the above we get
$$\varrho(\gamma_n,\delta)\ls\frac{C\sqrt{\ln (4/\delta)}}{\sqrt{n}},$$
where $C>0$ is an absolute constant.
\end{example}

We close this article with the example of the uniform measure on the Euclidean ball. It was proved in
\cite{Bonnet-Chasapis-Grote-Temesvari-Turchi-2019} that if $\epsilon\in (0,1)$ and $K_N={\rm conv}\{x_1,\ldots ,x_N\}$ where
$x_1,\ldots ,x_N$ are random points independently and uniformly chosen from $B_2^n$ then
$$\lim_{n\to \infty}\sup\left\{\frac{{\mathbb E} |K_N|}{|B_2^n|}:N\ls \exp\left((1-\epsilon)\left(\frac{n+1}{2}\right)\ln n\right)\right\}=0$$
and
$$\lim_{n\to \infty}\inf\left\{\frac{{\mathbb E} |K_N|}{|B_2^n|}:N\gr \exp\left((1+\epsilon)\left(\frac{n+1}{2}\right)\ln n\right)\right\}=1.$$
We shall obtain a similar conclusion with the approach of this work (the estimate below is in fact stronger since it sharpens the
width of the threshold from $O(1)$ to $O(1/\sqrt{n})$).

\begin{theorem}\label{th:ball}Let $D_n$ be the centered Euclidean ball of volume $1$ in ${\mathbb R}^n$. Then, the sequence
$\mu_n:=\mu_{D_n}$ exhibits a sharp threshold with $\varrho (\mu_n,\delta )\ls \frac{c}{\sqrt{\delta n}}$ and e.g. if $n$ is even then
we have that
$${\mathbb E}_{\mu_n}(\Lambda_{\mu_n}^{\ast })=\frac{(n+1)}{2}H_{\frac{n}{2}}+O(\sqrt{n})$$
as $n\to\infty $, where $H_m=\sum_{k=1}^m\frac{1}{k}$.
\end{theorem}

\begin{proof}Note that if $K$ is a centered convex body in ${\mathbb R}^n$ and $r>0$ then $\Lambda_{\mu_{rK}}^{\ast }(x)=\Lambda_{\mu_K}^{\ast }(x/r)$
for all $x\in {\mathbb R}^n$, where $\mu_{rK}$ is the uniform measure on $rK$. It follows that
$$\frac{1}{|rK|}\int_{rK}[\Lambda_{\mu_{rK}}^{\ast }(x)]^pdx=\frac{1}{|K|}\int_K[\Lambda_{\mu_K}^{\ast }(x)]^pdx$$
for every $p>0$ and $r>0$. This shows that in order to compute $\beta (\mu_{D_n)}$ it suffices to compute the ratio
$$\beta (\mu_{D_n})+1=\frac{\frac{1}{|B_2^n|}\int_{B_2^n}\Lambda^{*}(x)^2dx}{\left(\frac{1}{|B_2^n|}\int_{B_2^n}\Lambda^{*}(x)dx\right)^2}$$
where $\Lambda^{\ast }:=\Lambda_{\mu_{B_2^n}}^{\ast }$.
Having in mind Lemma~\ref{lem:body-beta-omega} we start by computing $\tau(\mu_{B_2^n})$. Set $\omega :=\omega_{\mu_{B_2^n}}$. Then,
$\omega (x)=\ln (1/\varphi(x))$  where $\varphi (x)=F(|x|)$,
$$F(r)=c_n\int_r^{1}(1-t^2)^{\frac{n-1}{2}}dt,\qquad r\in [0,1]$$ and
$c_n=\pi^{-1/2}\Gamma (\frac{n}{2}+1)/\Gamma(\frac{n+1}{2})$. From \cite[Lemma~2.2]{Bonnet-Chasapis-Grote-Temesvari-Turchi-2019}
we know that
$$F(r)=(1-r^2)^{\frac{n+1}{2}}h(r,n),$$
where
\begin{equation}\label{eq:hrn}\frac{1}{\sqrt{2\pi (n+2)}}\ls h(r,n)\ls\frac{1}{r\sqrt{2\pi n}}\end{equation}
for all $r\in (0,1]$. We assume that $n$ is even (the case where $n$ is odd can be treated in a similar way).
Using polar coordinates we compute
\begin{align*}
\frac{1}{|B_2^n|}\int_{B_2^n}\omega(x)\,dx&=-n\int_0^{1}r^{n-1}\ln(F(r))\, dr\\
&=-n\int_0^1 r^{n-1}\ln((1-r^2)^{\frac{n+1}{2}})\, dr-n\int_0^1 r^{n-1}\ln\left(h(r,n)\right)\, dr.
\end{align*}
The leading term is the first one and one can compute it explicitly. After making the change of variables $r^2=u$, we get
\begin{align}\label{eq:ball-1.1}
-n\int_0^1 r^{n-1}\ln((1-r^2)^{\frac{n+1}{2}})\, dr &=	-\frac{n(n+1)}{2}\int_0^1 r^{n-1}\ln(1-r^2)\, dr\\
\nonumber &=-\frac{n(n+1)}{4}\int_0^1u^{\frac{n-2}{2}}\ln(1-u)\,du=\frac{n(n+1)}{2n}H_{\frac{n}{2}},
\end{align}
using also Lemma~\ref{lem:digamma}. For the second term we recall from \eqref{eq:hrn} that $0\ls -\ln (h(r,n))\ls \tfrac{1}{2}\ln (2\pi (n+2))\ls c_1\ln n$, and hence
\begin{equation*}-n\int_0^1 r^{n-1}\ln\left(h(r,n)\right)\, dr\ls c_1\ln n\int_0^1nr^{n-1}dr=c_1\ln n.\end{equation*}
Therefore,
\begin{equation}\label{eq:ball-1}\frac{1}{|B_2^n|}\int_{B_2^n}\omega(x)\,dx=\frac{n+1}{2}H_{\frac{n}{2}}+O(\ln n).\end{equation}
Using again polar coordinates we write
\begin{align*}
\frac{1}{|B_2^n|}\int_{B_2^n}(\omega(x))^2dx &=n\int_0^{1}r^{n-1}\ln^2(F(r))\, dr\\
&=n\int_0^1 r^{n-1}\ln^2((1-r^2)^{\frac{n+1}{2}})\, dr+n\int_0^1 r^{n-1}\ln^2(h(r,n))\, dr\\
&\hspace*{1cm}+2n\int_0^1 r^{n-1}\ln((1-r^2)^{\frac{n+1}{2}})\ln(h(r,n))\, dr,
\end{align*}
As before, the leading term is the first one and we can compute it explicitly. After making the change of variables $r^2=u$, we get
\begin{align*}
n\left(\frac{n+1}{2}\right)^2\int_0^1 r^{n-1}\ln^2(1-r^2)\, dr&=n\frac{(n+1)^2}{8}\int_0^1u^{\frac{n-2}{2}}\ln^2(1-u)\,du\\
&=n\frac{(n+1)^2}{8}\left(\frac{2}{n}H_{\frac{n}{2}}^2-\frac{2}{n}\sum_{k=1}^{n/2}\frac{1}{k^2}\right).
\end{align*}
On the other hand, from \eqref{eq:hrn} we see that if $h(r,n)\ls 1$ then $0\ls -\ln (h(r,n))\ls \tfrac{1}{2}\ln (2\pi (n+2))\ls c_1\ln n$, while
if $h(r,n)>1$ then $0\ls \ln(h(r,n))\ls \ln (1/r)$. Therefore, $\ln^2(h(r,n))\ls c_2(\ln n)^2+\ln^2(r)$ for all $r\in (0,1]$. It follows that
$$n\int_0^1 r^{n-1}\ln^2(h(r,n))\, dr\ls c_3(\ln n)^2+n\int_0^1r^{n-1}\ln^2r\,dr\ls c_4(\ln n)^2.$$
Using again the fact that $\ln (h(r,n)^{-1})\ls c_1\ln n$ as well as \eqref{eq:ball-1.1}, we check that
$$2n\int_0^1 r^{n-1}\ln((1-r^2)^{\frac{n+1}{2}})\ln(h(r,n))\, dr\ls \frac{n(n+1)}{2n}H_{\frac{n}{2}}\cdot c_1\ln n\ls c_5n(\ln n)^2.$$
From these estimates we have
\begin{equation}\label{eq:ball-2}\frac{1}{|B_2^n|}\int_{B_2^n}(\omega(x))^2\,dx=\frac{(n+1)^2}{4}H^2_{\frac{n}{2}}+O(n(\ln n)^2).\end{equation}
From \eqref{eq:ball-1} and \eqref{eq:ball-2} we finally get
$$\tau(\mu_{B_2^n})=O\left(\frac{n(\ln n)^2}{n^2H^2_{\frac{n}{2}}}\right)=O(1/n).$$
Then, Lemma~\ref{lem:body-beta-omega} and a simple computation show that
$$\beta (\mu_{D_n})=\left(\tau(\mu_{B_2^n})+O(L_{\mu_{B_2^n}}^2/\sqrt{n})\right)\left(1+O(L_{\mu_{B_2^n}}^2/\sqrt{n})\right)
=O(1/\sqrt{n}),$$
because $L_{\mu_{B_2^n}}\approx 1$. Finally, note that by the estimate \eqref{eq:L*-omega} in Corollary~\ref{cor:L*-omega} we have
$${\mathbb E}_{\mu_n}(\Lambda_{\mu_n}^{\ast })=\frac{1}{|B_2^n|}\int_{B_2^n}\omega(x)\,dx+O(\sqrt{n})
=\frac{(n+1)}{2}H_{\frac{n}{2}}+O(\sqrt{n})$$
as $n\to\infty $.
\end{proof}

\begin{note*}\rm The above discussion leaves open the following basic question: to estimate
$$\beta_n^{\ast }:=\sup\{\beta(\mu_K):K\;\hbox{is a centered convex body of volume}\;1\;\hbox{in}\;{\mathbb R}^n\}$$
or, more generally,
$$\beta_n:=\sup\{\beta(\mu):\mu\;\hbox{is a centered log-concave probability measure on}\;{\mathbb R}^n\}.$$
\end{note*}

\bigskip

\noindent {\bf Acknowledgements.} The authors are grateful to the referee for very useful comments and suggestions on the
presentation of the results of this article. The first and second author acknowledge support by the Hellenic Foundation for
Research and Innovation (H.F.R.I.) under the ``First Call for H.F.R.I. Research Projects to support Faculty members and Researchers and
the procurement of high-cost research equipment grant" (Project Number: 1849). The third named author acknowledges support by the Hellenic Foundation for
Research and Innovation (H.F.R.I.) under the ``Third Call for H.F.R.I. PhD Fellowships" (Fellowship Number: 5779).

\bigskip

\footnotesize

\bibliographystyle{amsplain}


\medskip

\thanks{\noindent {\bf Keywords:} log-concave probability measures, half-space depth, isotropic constant, random polytopes, convex bodies.

\smallskip

\thanks{\noindent {\bf 2010 MSC:} Primary 60D05; Secondary 60E15, 62H05, 52A22, 52A23.}

\bigskip

\bigskip

\noindent \textsc{Silouanos \ Brazitikos}: Department of
Mathematics, National and Kapodistrian University of Athens, Panepistimioupolis 157-84,
Athens, Greece.

\smallskip

\noindent \textit{E-mail:} \texttt{silouanb@math.uoa.gr}

\bigskip

\noindent \textsc{Apostolos \ Giannopoulos}: Department of
Mathematics, National and Kapodistrian University of Athens, Panepistimioupolis 157-84,
Athens, Greece.

\smallskip

\noindent \textit{E-mail:} \texttt{apgiannop@math.uoa.gr}

\bigskip

\noindent \textsc{Minas \ Pafis}: Department of
Mathematics, National and Kapodistrian University of Athens, Panepistimioupolis 157-84,
Athens, Greece.

\smallskip

\noindent \textit{E-mail:} \texttt{mipafis@math.uoa.gr}

\bigskip

\end{document}